\documentclass[11pt,oneside]{amsart}
\usepackage[T1]{fontenc}
\usepackage[latin9]{inputenc}
\usepackage{geometry}
\usepackage{verbatim}
\usepackage{textcomp}
\usepackage{amsthm}
\usepackage{amstext}
\usepackage[foot]{amsaddr} 
\usepackage{enumerate}
\usepackage{stmaryrd}
\usepackage{amssymb}
\usepackage{latexsym}
\usepackage{mathtools}
\usepackage{esint}
\usepackage{xcolor}
\usepackage{mathrsfs}
\usepackage{wasysym}
\usepackage{nomencl}
\usepackage[english]{babel}
\usepackage{tikz-cd}
\usepackage{pdfpages}
\usepackage{hyperref}
\usepackage[shortlabels]{enumitem}
\def\bZ{{\mathbb Z}}

\def\bZ{{\mathbb{Z}}}
\def\bN{{\mathbb{N}}}

\def\cA{{\mathcal A}}

\def\cU{{\mathcal U}}

\def\cK{{\mathcal K}}
\def\cM{{\mathcal M}}
\def\cN{{\mathcal N}}

\def\cC{{\mathcal C}}
\def\cJ{{\mathcal J}}
\def\cD{{\mathcal D}}

\def\g{{\mathfrak g}}

\def\h{{\mathfrak h}}

\def\bb{\bullet \bullet}
\def\b{\bullet}

\def\ov\g{{\overline{\mathfrak{g}}}}

\makeatother

\newtheorem{theorem}{Theorem}[section]
\newtheorem*{theorem*}{Theorem}
\newtheorem{proposition}[theorem]{Proposition}
\newtheorem{corollary}[theorem]{Corollary}
\newtheorem{lemma}[theorem]{Lemma}

\theoremstyle{definition}
\newtheorem{definition}[theorem]{Definition}

\theoremstyle{definition}
\newtheorem{remark}[theorem]{Remark}

\theoremstyle{definition}

\title{On homological reduction of Poisson structures}

\author{Pedro H. Carvalho}
\address{Instituto de Matem\'atica Pura e Aplicada - IMPA}
\email{pedro.carvalho@impa.br} 
\begin{document}

\maketitle

\begin{abstract}
    Given a $\g$-action on a Poisson manifold $(M, \pi)$ and an equivariant map $J: M \rightarrow \h^*,$ for $\h$ a $\g$-module, we obtain, under natural compatibility and regularity conditions previously considered by Cattaneo-Zambon, a homotopy Poisson algebra generalizing the classical BFV algebra  described by Kostant-Sternberg in the usual hamiltonian setting. As an application of our methods, we also derive homological models for the reduced spaces associated to quasi-Poisson and hamiltonian quasi-Poisson spaces.

\end{abstract}

\tableofcontents

\section{Introduction}

Homological methods in symplectic and Poisson geometry can be traced back to the physics papers on quantization of mechanical systems with symmetries (\cite{BRS, BRS1, BRS2, BV, BV1, BV2, BF, BV3}). Subsequently, Kostant-Sternberg \cite{Kostant}, Stasheff \cite{Stasheff} and Kimura \cite{Kimura} unraveled the mathematical content of the constructions that physicists had considered. Nowadays, these techniques, which go by the name of BFV formalism, are well-established and largely used in mathematical physics, being a fundamental tool, e.g., for the treatment of quantization of field theories.

For a  Poisson manifold $(M, \pi)$ endowed with a hamiltonian $\g$-action with moment map $J: M \rightarrow \g^*,$ the BFV algebra, called the BRS algebra in \cite{Kostant}, is a differential super Poisson algebra whose cohomology in degree zero recovers the Poisson algebra of the reduced space associated to the level set $J^{-1}(0),$ for $0 \in \g^*$ a regular value. In this sense, the BFV construction produces homological models for the classical Marsden-Weinstein quotients (\cite{MarsdenRatiu, MarsdenWeinstein}).\par

More general reduction setups for Poisson structures were explored by Marsden-Ratiu (\cite{MarsdenRatiu}) and by Falceto-Zambon (\cite{FalcetoZambon}). In these cases, the input for reduction consists of a submanifold endowed with integrable distribution compatible with a Poisson structure in a suitable way. In \cite{Sharapov}, Lyakhovich-Sharapov considered this, under the name of generic gauge systems, to perform reduction of what they called weak Poisson structures.
In \cite{CattaneoZambon}, on the other hand, Cattaneo-Zambon concretely formulated an extension of the usual hamiltonian setting for the reduction of a Poisson manifold $(M, \pi)$ by considering   a $\g$-action $\psi: \g \rightarrow \mathfrak{X}(M)$ and an equivariant map $J: M \rightarrow \h^*,$ for $\h$ a $\g$-module. In this context, the pair $(\psi, J)$ is said to be 
\begin{itemize}
    \item \textit{regular}, if $0 \in \h^*$ is a regular value of $J: M \rightarrow \h^*$ and the $\g$-action along $C \coloneqq J^{-1}(0)$ integrates to a free and proper action of a Lie group $G;$ 
    \item \textit{compatible} with the Poisson structure $\pi,$ if the space of functions on $M$  whose restrictions to $C$ are $\g$-invariant is closed  
    under the Poisson bracket induced $\pi$ and, moreover, the distribution $\cD  \coloneqq \langle \psi(u) \rangle_{u \in \g} \subset TC$ contains $\pi^{\sharp}(Ann (TC)).$  
\end{itemize}

For a regular and compatible pair $(\psi, J)$ on a Poisson manifold $(M, \pi),$ Cattaneo-Zambon have proved that the reduced space $C_{red} \coloneqq C/G$ inherits a Poisson structure $\pi_{red},$ thereby obtaning a  generalized version of the classical Marsden-Weinstein theorem. The main result of this note provides  the corresponding homological model for this construction, which turns out to be a homotopy version of the BFV algebra described by Kostant-Sternberg in the usual hamiltonian setting (\cite{Kostant}). More precisely, we shall prove the following.

\begin{theorem}
   Let $\psi: \g \rightarrow \mathfrak{X}(M)$ be a Lie algebra action on a Poisson manifold $(M, \pi),$ and let $J: M \rightarrow \h^*$ be a $\g$-equivariant map, for $\h$  a $\g$-module. Assume that the reduction data $(\psi, J)$ is regular and compatible with the Poisson structure $\pi.$  Then, the $\bZ$-graded algebra $$\mathcal{K}^\bullet_{\g, \h} \coloneqq C^{\infty}(M) \otimes \bigwedge^\bullet \g^* \otimes \bigwedge^\bullet \h ,$$ graded by total ghost number, admits a homotopy Poisson structure with differential $\partial: \mathcal{K}^\bullet_{\g, \h} \rightarrow \mathcal{K}^{\bullet + 1}_{\g, \h}$ and a sequence of $k$-ary brackets $\{\cdot, \dots, \cdot\}_k,$ $k \geq 2,$ such that the Poisson algebra $(H_{\partial}^0(\mathcal{K^\bullet_{\g, \h}}), \{\cdot, \cdot\}_2)$ is identified with the reduced Poisson algebra $(C^{\infty}( C_{red}), \{\cdot, \cdot\}_{\pi_{red}}).$
\end{theorem}

To prove this statement, we perform some steps. Just as the reduction scheme of Cattaneo-Zambon follows from interpreting the Marsden-Weinstein theorem for the shifted cotangent bundle in classical (non-graded) terms, our first step is the construction of a BFV algebra for the reduction of the shifted cotangent bundle $\cM \coloneqq T^{*}[1]M$ by the hamiltonian action of a graded Lie algebra $\ov\g = \h[1]\oplus \g,$ which is a homological model for the reduction of multivector fields and is independent of any Poisson structure.

Following the BFV procedure, this is obtained from the $\bZ$-graded symplectic manifold $\cN \coloneqq \cM \times T^{*}[1]\ov\g^*[-1]$ and a special function $Q_{\ov\g} \in C^{\infty}(\cN),$ 
called the {\em BRST charge}, whose corresponding  (homological) hamiltonian vector field $\{Q_{\ov\g}, \, \cdot \, \}$ encodes all the gauge symmetries. We shall denote the degree one BFV algebra by $(\cA, \{Q_{\ov\g}, \, \cdot \, \}),$ where $\cA \coloneqq C^{\infty}(\cN).$  We observe that, upon the regularity conditions we assume, the construction of the BFV algebra $(\cA, \{Q_{\ov\g}, \cdot \})$ neither require the use of homological pertubation methods nor a further enlargement of the ghost/antighost spectrum (see Remark \ref{remark: BFV in general}). 

As for the second step, exploring the close relation between the graded reduction data and the reduction data for Poisson structures (\cite{CattaneoZambon}, \cite{Sharapov}), we show that the BFV model $(\cA, \{Q_{\ov\g}, \, \cdot \, \})$ can be naturally coupled with a compatible Poisson structure (\cite{Sharapov}, \cite{AleZabzine}). This produces an extended BFV model in which both the reduction data and the Poisson structure are brought together as parts of a single object, which we shall call {\em extended BRST charge} and denote by  $S^{(\infty)}.$
 
 The final step consists in finding the algebra $\mathcal{K}^\bullet_{\g, \h}$ as an abelian subalgebra of the BFV algebra -- in graded-geometric terms, it can be seen as the algebra of functions of a lagrangian submanifold of $\cN$ -- and we take the differential $\partial: \mathcal{K}^\bullet_{\g, \h} \rightarrow \mathcal{K}^{\bullet + 1}_{\g, \h}$ and the higher $k$-ary brackets $\{\cdot, \dots, \cdot\}_k,$ $k \geq 2,$ on $\mathcal{K}^\bullet_{\g, \h}$ to be derived brackets defined in terms of the extended BRST charge $S^{(\infty)}.$ We conclude that this endows $\mathcal{K}^\bullet_{\g, \h}$
with  a homotopy Poisson structure  by applying Voronov's  construction of homotopy algebras (\cite{Voronov}).  \par 

When the reduction data $(\psi, J)$ are defined by a usual hamiltonian action, that is, $\psi: \g \rightarrow \mathfrak{X}(M)$ a $\g$-action on $(M, \pi)$ admitting an equivariant moment map $J: M \rightarrow \g^*,$ our theorem recovers the homological model obtained by Kostant-Sternberg (\cite{Kostant}). The fact that in this case we obtain a differential Poisson algebra, rather than a more general homotopy Poisson algebra, happens more generally whenever the compatible pair $(\psi, J)$ is determined by the hamiltonian action of a differential graded Lie algebra $( \ov\g, \delta)$ on $\cM.$ Moreover, since our general construction applies to ``weak'' Poisson bivectors (\cite{Sharapov}), we also obtain homological models for the reduced spaces associated to quasi-Poisson spaces (Lie bialgebra and quasi-bialgebra actions, \cite{Lu}, \cite{Alekseev2}) and hamiltonian quasi-Poisson spaces (group-valued moment maps, \cite{Alekseev}). 

This paper illustrates the explicit role of homotopy Poisson structures in Poisson reduction and indicates that reduction of other classical geometric structures naturally described in the terms of graded manifolds, notably Courant algebroids (\cite{Roytenberg, Henrique, GGGR}), may be interpreted from an analogous homotopy standpoint.
 
The paper is organized as follows. In $\S$\ref{section: symplectic degree one}, we briefly recall graded symplectic degree one manifolds (i.e., shifted cotangent bundles) together with their canonical symplectic structure. In $\S$\ref{section: hamiltonian reduction}, we introduce the hamiltonian reduction of the shifted cotangent bundle $\cM \coloneqq T^*[1]M,$ recall how its interpretation in non-graded terms leads to the Cattaneo-Zambon generalized hamiltonian context, 
and discuss the reducibility of weak Poisson structures with respect to such data. In $\S$\ref{section brst in degree zero}, for the purpose of motivation, we briefly recall the  BFV construction in the context of Kostant-Sternberg (\cite{Kostant}) and its graded-geometric interpretation. In $\S$\ref{section: brst in degree one}, we construct the BFV model for the hamiltonian reduction of the shifted cotangent bundle $\cM \coloneqq T^{*}[1]M.$ 
In $\S$\ref{section: the homotopy nature of Poisson reduction}, we show that
this homological model for degree one reduction can be naturally coupled with a given weak Poisson structure; then we explain how this gives rise to a homotopy Poisson structure encoding the corresponding Poisson reduction and, consequently, provides the homotopy generalization of the Kostant-Sternberg result in the Cattaneo-Zambon generalized hamiltonian context. To conclude, we analyze the classical hamiltonian reduction of Poisson structures, quasi-Poisson spaces and hamiltonian quasi-Poisson spaces through the lens of the degree one BFV construction.

\vspace{0.2cm} \noindent 
\textbf{Acknowlegdments:} 
The author thanks H. Bursztyn for his continual support and for all the insightful discussions and helpful suggestions; A. Cabrera for valuable comments on a preliminary version of this note; A. S. Cattaneo for the stay at U. Z\"urich, where part of this work was done; J. Stasheff, S. L. Lyakhovich, A. A. Sharapov for the feedback; and the CNPq and IMPA for the financial support.

\section{Graded symplectic degree one manifolds}\label{section: symplectic degree one}

In this section, we briefly recall the shifted cotangent bundle of a smooth manifold $M$ and its canonical symplectic structure.

\vspace{0.3cm}

Let $M$ be a smooth manifold, and let $E \rightarrow M$ be a vector bundle over $M.$
The degree one manifold $E^*[1]$ is the manifold $M$ endowed with the sheaf of graded algebras $\Gamma(\bigwedge ^\bullet E),$ which we refer to as the sheaf of functions of $E^*[1];$  for $U \subset M$ open, we let  $C^{\infty}(E^{*}[1]) (U) \coloneqq \Gamma(\bigwedge ^\bullet E|_U).$ In particular, for the tangent bundle $TM \rightarrow M,$ we obtain the {\em shifted cotangent bundle} of $M$, which we denote by $T^*[1]M.$ In this case, the corresponding sheaf of functions is given by the multivector fields on $M,$ 
$$
C^\infty(T^*[1]M) \coloneqq \Gamma(\bigwedge^\bullet TM) = \mathfrak{X}^\bullet(M).$$ 

In local coordinates $(U, x_1, \dots, x_n)$ for $M,$ a multivector field $X \in \mathfrak{X}^p(M)$ can be expressed as $$X = \sum_{i_1 < \dots < i_p} a_{i_1 \dots i_p} \partial_{x_{i_1}} \wedge \dots \wedge \partial_{x_{i_p}},  \hspace{0.2cm} a_{i_1 \dots i_p} \in C^{\infty}(M)|_U, $$ so if we introduce $\xi_{i_j} \coloneqq \partial_{x_{i_j}}$ as new anticommuting variables, it can be rewritten as $$X  = \sum_{i_1 < \dots < i_p} a_{i_1 \dots i_p} \xi_{i_1}  \dots \xi_{i_p}, \hspace{0.1cm} a_{i_1 \dots i_p} \in C^{\infty}(M)|_U;$$ hence, locally, the algebra of functions of $T^{*}[1]M$ is simply the graded symmetric algebra generated by the degree one (anticommuting) variables $\{\xi_i, i = 1, \dots, n \}$ and the degree zero (commuting) variables $\{x_i, i = 1, \dots, n\}.$\par

Shifted cotangent bundles are the canonical examples of graded symplectic degree one manifolds; and, as observed by Roytenberg in \cite{Roytenberg}, they are the unique examples of such a structure. 
In fact, the standard Lie bracket of vector fields can be extended, as a derivation of degree $-1,$ to the whole algebra of multivector fields $\mathfrak{X}^\bullet(M)$ and, in this way, we obtain the Schouten bracket of multivector fields (\cite{Dufour}). Hence, $(\mathfrak{X}^\bullet(M), \{\cdot, \cdot\})$ is a Gerstenhaber algebra, i.e., 
a degree $-1$ Poisson algebra.
For $X \in \mathfrak{X}^{p}(M)$ and $Y \in \mathfrak{X}^q(M),$ the shifted Poisson bracket $\{X, Y\} \in \mathfrak{X}^{p+q-1}(M)$ is given, in local coordinates, by

$$ \{X, Y\} = \sum \frac{\partial X}{\partial \xi_i}\frac{\partial Y}{\partial x_i} - (-1)^{(p-1)(q-1)} \sum_i \frac{\partial Y}{\partial \xi_i}\frac{\partial X}{\partial x_i},$$
which is an expression analogous to that defining the canonical Poisson bracket on $T^*M.$

In graded terms, a Poisson structure on $M,$ that is, a bivector $\pi \in \mathfrak{X}^2(M)$ such that $\{\pi, \pi\} = 0,$ is simply a special degree two function on $T^*[1]M.$ This point of view to Poisson structures will be important, because, as we shall see, certain submanifolds of the shifted cotangent bundle can be interpreted as providing reduction data for a given Poisson structure on $M$ (\cite{CattaneoZambon}).

\section{Hamiltonian Reduction}\label{section: hamiltonian reduction}
In this section, we discuss the hamiltonian reduction of the shifted cotangent bundle $\cM \coloneqq T^{*}[1]M$ and show how it provides a natural setting for performing reduction of Poisson structures. This is all based on \cite{CattaneoZambon}.
\subsection{Moment maps in degree one} \label{subsetion: moment map in degree one} Let $(\cM, \{\,\cdot, \cdot\, \})$ be the shifted cotangent bundle endowed with its natural degree $-1$ Poisson structure, and let $\ov\g \coloneqq \h[1] \oplus \g$ 
be a graded Lie algebra concentrated in degrees $-1, 0$. An infinitesimal action of $\ov\g$ on $\cM $ is given by a morphism $\Psi: \ov\g \rightarrow \mathfrak{X}(\cM)$ of graded Lie algebras, where $\mathfrak{X}(\cM)$ is the Lie algebra of derivations of $C^{\infty}(\cM).$ 

A moment map for the infinitesimal action $\Psi: \ov\g \rightarrow \mathfrak{X}(\cM)$ is a morphism of (odd) Lie algebras $J^{\sharp}: \ov\g[-1] \rightarrow C^{\infty}(\cM)$ 
satisfying $u_{\cM} = \{J^{\sharp}_1(u),\, \cdot \, \},$ for $u \in \g,$  and $ v_{\cM} = \{ J^{\sharp}_0(v), \, \cdot \, \},$ for $v \in \h.$ From a dual perspective, the map $J^{\sharp}: \ov\g[-1] \rightarrow C^{\infty}(\cM)$ can be seen as the sheaf morphism of a map of degree one manifolds  $(J, J^\sharp): \cM \rightarrow (\ov\g[-1])^*,$
where $(\ov\g[-1])^*$ denotes the degree one manifold associated to the vector bundle $\h^* \times \g \rightarrow \h^*$ and $J: M \rightarrow \h^*$ is the corresponding body map.

This graded hamiltonian setting admits an interpretation in non-graded terms given by the following one-to-one correspondence:
\vspace{0.1cm}
\begin{center}
    \begin{tabular}{|c|c|}
       \hline
       graded $\ov\g$-action $\Psi: \ov\g \rightarrow \mathfrak{X}(\cM)$ & $\g$-action $\psi: \g \rightarrow \mathfrak{X}(M)$  with  \\
        with moment map $ J^\sharp: \ov\g[-1] \rightarrow C^{\infty}(\cM)$  & $J: M \rightarrow \h^* $ equivariant, for $\h$ a $\g$-module \\
        \hline
    \end{tabular}
\end{center}
 \vspace{0.1cm}
 
Later, we shall see how this correspondence provides a generalization of the classical hamiltonian setup for the reduction of Poisson structures ($\S$\ref{subsection: reducible Poisson str}). 

\subsection{Constraint submanifolds} 

To perform reduction, we look at level sets of the moment map $(J, J^{\sharp})$. The level set $\cC \coloneqq (J, J^\sharp)^{-1}(0)$ is defined by the sheaf of vanishing ideals $\cJ$ generated by the image of the map $J^\sharp: \ov\g[-1] \rightarrow C^{\infty}(\cM).$ As in the case of smooth manifolds, if $0 \in \h^*$ is a regular value  of  $(J, J^\sharp): \cM \rightarrow (\ov\g[-1])^*,$ this defines a smooth submanifold of $\cM$  (\cite[Cor.~2.24]{GGGR}).

The zero set $Z(\cJ)$ of the ideal $\cJ$ in $M$ is the vanishing set of the ideal spanned by the functions $\{J^*(v) : v \in \h \subset C^{\infty}(\h^*) \},$  which is the zero level set  $J^{-1}(0)$ 
of the $\g$-equivariant map $J: M \rightarrow \h^*.$
On the other hand, the fact that $J^\sharp: \ov\g[-1] \rightarrow C^{\infty}(\cM)$  is bracket preserving implies $\cC$ is coisotropic, which means that $\{\cJ, \cJ\} \subset \cJ.$ This condition applied to the generators of the ideal $\cJ$ shows that it codifies, geometrically, the level set $C \coloneqq J^{-1}(0)$ 
endowed with the tangent involutive distribution $\cD$ spanned by $\{J_1^\sharp(u)\}_{u \in \g}.$

Following \cite{GGGR}, we know that $0 \in \h^*$ being a regular value of $(J, J^{\sharp}): \cM \rightarrow (\ov\g[-1])^*$ is equivalent to it being a regular value of $J: M \rightarrow \h^*$ and the action  $\psi \coloneqq J^\sharp|_{\g} : \g \rightarrow \mathfrak{X}(M)$  being free along $C \coloneqq J^{-1}(0).$
We will say that the pair $(\psi, J)$ is \textit{regular} if the free action  $\psi: \g \rightarrow \mathfrak{X}(\cM)$ on $C$ integrates to a free and proper action of a Lie group $G$; in this case, we have a smooth reduced space $C_{red} \coloneqq C/G.$ 
\par 
The algebra of functions on the graded submanifold $\cC \coloneqq (J, J^{\sharp})^{-1}(0)$ is obtained by moding out $C^{\infty}(\cM)$ by $\cJ$, i.e., \begin{equation}\label{functions on the constraint graded submanifold}
C^{\infty}(\cC) \cong \frac{C^{\infty}(\cM)}{\cJ};
\end{equation}   
so, $ \cC = Ann(\mathcal{D})[1],$ for the vector bundle $Ann(\mathcal{D}) \rightarrow C,$ where  $Ann(\mathcal{D})$ denotes the annihilator of the distribution $\mathcal{D}.$ \par 

\begin{remark}\label{generic gauge system - coisotropic submanifolds}
As it was shown in \cite{CattaneoZambon}, the geometric data associated to the vanishing ideal of a general coisotropic submanifold of dimension $(r, s)$ of $\cM \coloneqq T^{*}[1]M$  consists of an embedded submanifold of $M$ of codimension $r$ endowed with an integrable tangent distribution of rank $s.$ So the generic gauge systems considered by Lyakhovich-Sharapov in \cite{Sharapov} provide examples of smooth coisotropic submanifolds of $\cM;$ actually, what they consider can be identified as the degree one version of the so called irreducible first class constraints.
\end{remark}
\subsection{Reduction} In constructing the reduced space associated to the constraint submanifold $\cC \coloneqq (J, J^{\sharp})^{-1}(0)$ (or to any coisotropic submanifold of $\cM$), it is important to bear in mind how the functions there are to be obtained: first, we restrict the functions on $\cM$ to the constraint submanifold $\cC,$ which can be described by the quotient in (\ref{functions on the constraint graded submanifold}), and then we take gauge invariance, i.e., invariance with respect to the infinitesimal action of $\ov\g.$ In other words, if the reduced space exists and is denoted by $\cC_{red},$ then its algebra of functions  is  given by 
\begin{equation}\label{gauge invariant function on the constraint submanifold}
C^{\infty}(\cC_{red}) \cong \bigg(\frac{C^{\infty}(\cM)}{\cJ}\bigg)^{\Psi}, 
\end{equation}
where the right-hand side is to be read as gauge invariant functions on the constraint submanifold $\cC.$
Notice that if $$N(\cJ) = \{ f \in C^{\infty}(\cM) \, \, | \, \, \{f, \cJ\} \subset \cJ \} $$ is the normalizer of $\cJ$ in $C^{\infty}(\cM),$ then we have the natural isomorphism 
\begin{equation}\label{gauge invariant function/normalizer}
 \bigg(\frac{C^{\infty}(\cM)}{\cJ}\bigg)^{\Psi} \cong \frac{N(\cJ)}{\cJ}. 
\end{equation}
 
In degree zero, the gauge invariance condition gives functions that are constant along the foliation 
integrating the distribution $\cD \coloneqq \langle J_1^\sharp(u) \rangle_{u \in \g}$, whereas, in degree one, it gives vector fields in the normalizer of this distribution (i.e., projectable vector fields).  
If the pair $(\psi, J)$ associated to the moment map $J^{\sharp}: \ov\g[-1] \rightarrow C^{\infty}(\cM)$ is regular, then the leaf space of $C$ is given by $ C_{red} \coloneqq C/G.$  Moreover, as a special case of \cite[Prop.~5.11]{CattaneoZambon}, one obtains that the algebra of gauge invariant functions on the constraint submanifold $\cC$ is generated in degrees zero and one, so it follows that the reduced space $\cC_{red}$ is naturally isomorphic to $T^{*}[1]C_{red}$ -- all this is also valid for a general coisotropic submanifold, if one can guarantee the smoothness of the corresponding quotient. This gives the following version of the Marsden-Weinstein reduction theorem (see \cite[Props.~4.2 and ~10.1]{CattaneoZambon}).

\begin{theorem}\label{marsden-weinstein cattaneozambon}
Let $J^\sharp: \ov\g[-1] \rightarrow C^{\infty}(\cM)$ be a moment map for an infinitesimal action of $\ov\g \coloneqq \h[1]\oplus \g $ on $\cM.$ Assume that $0 \in \h^*$ is a regular value of $J: M \rightarrow \h^*$ and that the action $\psi \coloneqq J^{\sharp}|_{\g}: \g \rightarrow \mathfrak{X}(M)$ on $C \coloneqq J^{-1}(0)$ integrates to a free and proper action of a Lie group $G$, that is, assume that the pair $(\psi, J)$ is regular. Then the corresponding degree one reduced space $\cC_{red}$ exists and  is naturally isomorphic to $T^*[1]\big(C/G\big).$ 
\end{theorem}

\subsection{Reducible Poisson structures}\label{subsection: reducible Poisson str}
Let $\cJ \subset C^{\infty}(\cM)$ be the ideal defined by the moment map $J^{\sharp}: \ov\g[-1] \rightarrow C^{\infty}(\cM).$ Due to (\ref{gauge invariant function/normalizer}), we say that Poisson structure $\pi \in C_2^{\infty}(\cM) = \mathfrak{X}^2(M)$ is reducible if $\pi \in N(\cJ);$ in this case, it gives rise to a function $\pi_{red} \in  C_2^{\infty}(\cC_{red}),$ which is a Poisson structure on $C_{red}.$ Notice, however, that for a bivector $\pi \in   C_2^{\infty}(\cM)$ to induce a Poisson structure on $C_{red},$ it suffices that $\{\pi, \pi \} \in \cJ;$ following \cite{Sharapov}, we call bivectors satisfying this last condition \textit{weak Poisson structures}. \par

Observe that, if we set 
$$C^{\infty}(M)|_{\g\text{-inv}} \coloneqq \{f \in C^{\infty}(M) \,  : \, f|_{C \coloneqq J^{-1}(0)} \hspace{0.2cm} \text{is $\g$-invariant} \},
$$ 
then, for a Poisson structure $\pi \in \mathfrak{X}^2(M)$, the condition $\pi \in N(\cJ)$ is equivalent to $C^{\infty}(M)|_{\g\text{-inv}}$ being a Poisson subalgebra of $C^{\infty}(M)$ and $\pi^\sharp(Ann(TC)) \subset \cD \coloneqq \langle J_1^\sharp(u) \rangle_{u \in \g_0}.$ Based on this and on the hamiltonian correspondence of $\S$\ref{subsetion: moment map in degree one}, we introduce the following definition.
 
\begin{definition} \label{compatibility of the Poisson struct with the reduction data (psi, J)}
Let $\psi: \g \rightarrow \mathfrak{X}(M)$ be a Lie algebra action on a Poisson manifold $(M, \pi),$ and let $J: M \rightarrow \h^*$ be a $\g$-equivariant map, for $\h$ a $\g$-module. The pair $(\psi, J)$ is said to be  compatible  with the Poisson structure $\pi$  if 
\begin{enumerate}
    \item $C^{\infty}(M)|_{\g\text{-inv}} \subset C^{\infty}(M)$ is a Poisson subalgebra;
    \item $\pi^\sharp(Ann(TC)) \subset \cD \coloneqq \langle \psi(u) \rangle_{u \in \g}.$
\end{enumerate}
 \end{definition}

 A usual hamiltonian $\g$-action $\psi: \g \rightarrow \mathfrak{X}(M)$ on a Poisson manifold $(M, \pi)$ with a moment map $J: M \rightarrow \g^*$ fits in the above definition: from $\pi^\sharp(d(J^*u)) = \psi(u), u \in \g, $ it follows that $\pi^\sharp(Ann(TC)) = \langle \psi(u) \rangle_{u \in \g} \subset TC;$ and this, together with the invariance of $\pi$ (that is, $\mathscr{L}_{\psi(u)} \pi = 0, u \in \g$), gives condition (1).  Therefore, Definition \ref{compatibility of the Poisson struct with the reduction data (psi, J)} clarifies the sense in which the graded moment map $J^\sharp: \ov\g[-1] \rightarrow C^{\infty}(\cM)$ generalizes the classical hamiltonian setup for the reduction of a Poisson manifold $(M, \pi)$. Observe that, in general, the Poisson structure $\pi$ is no longer necessarily invariant.
 
 From this perspective, the degree one Marsden-Weinstein theorem of Cattaneo-Zambon (Theorem \ref{marsden-weinstein cattaneozambon}) implies the following generalized version of the classical Marsden-Weinstein reduction theorem (\cite[Props.~4.2 and~10.1]{CattaneoZambon}).

\begin{theorem}\label{generalized context - Marsden Weinstein thm}
Let $\psi: \g \rightarrow \mathfrak{X}(M)$ be a Lie algebra action on a Poisson manifold $(M, \pi),$ and let $J: M \rightarrow \h^*$ be a $\g$-equivariant map, for $\h$ a $\g$-module. Assume that the pair $(\psi, J)$ is regular and compatible with the Poisson structure $\pi.$ Then the quotient $C_{red} \coloneqq C/G$ inherits a Poisson structure $\pi_{red}.$
\end{theorem}

We will now proceed to provide the corresponding homological version of the above theorem.
\section{Review of BFV in degree zero}\label{section brst in degree zero}

In this section, for the purpose of motivation, we briefly recall the BFV construction in the context of   hamiltonian $\g$-actions on Poisson manifolds. We refer the reader to \cite{Kostant} for further details regarding this case (in \cite{Stasheff}, a more general setup is considered).
\vspace{0.3cm}

The BFV algebra provides a homological version of the reduction of hamiltonian spaces: it consists of a complex whose differential encodes the hamiltonian symmetries and whose cohomology models the reduced Poisson algebra.
 
Given a Poisson manifold $(M, \pi)$ and $J: M \rightarrow \g^*$ a moment map for a hamiltonian $\g$-action on $M,$ the BFV algebra is defined by the following data: 

\begin{itemize}
\item  
the algebra $\mathcal{K}^\bullet \coloneqq \bigoplus_{n \in \mathbb{Z}} \mathcal{K}^n,$ where
$$\mathcal{K}^n \coloneqq \bigoplus_{n = p - q} K^{p, q},$$
with 
\begin{equation}\label{terms of the bigraded vsp K}
 K^{p, q} \coloneqq C^{\infty}(M) \otimes \bigwedge^p \g^* \otimes \bigwedge ^q \g,
\end{equation}
endowed with the degree zero (super) Poisson bracket obtained from the Poisson bracket on $C^{\infty}(M),$ the standard pairing between $\g$ and $\g^*$ and the condition $\{f, u^*\} \equiv 0,$ for $f \in C^{\infty}(M)$ and $u^* \in \g^*;$ 

\item the degree one element $Q_{\g} \in \mathcal{K}^{1}$ given by \begin{equation}\label{degree zero brst charge}
Q_{\g} = J_i\otimes u^{*}_i - \frac{1}{2} c_k^{ij}\, u^{*}_{i}u^{*}_{j}u^k,
\end{equation}
where  $c_k^{ij}$ are structural constants associated to the basis $\{u^i\},$ and $J_i$ denotes the $i$th-component of the moment map $J: M \rightarrow \g^\ast$, 
that is, $J_i = u^i \circ J,$
which satisfies  \begin{equation}\label{master eq deg zero brst charge}
\{Q_{\g}, Q_{\g}\} = 0.
\end{equation}
\end{itemize}

Then the BFV algebra is the differential graded Poisson algebra $(\mathcal{K}^{\bullet}, \{Q_{\g}, \cdot \, \}).$ Usually, the gradings of $K^{\bb}$ associated to $\g$ and $\g^*$ are called {\em antighost} and {\em ghost} numbers, respectively, whereas the grading of $\cK^{\bullet}$ is called {\em total ghost number}. 
The distinguished element in (\ref{degree zero brst charge}) and the equation (\ref{master eq deg zero brst charge}) that it satisfies are often referred to as the {\em BRST charge} and the {\em master equation}, respectively, whereas the bracket $\{\cdot, \cdot\}$ is called the {\em BFV bracket.}

Reduction involves two steps: restricting to a level set of the moment map and taking gauge invariance. The BFV algebra comes out of realizing this two-step process homologically. 
For the restriction part, we consider the so called  Koszul complex
\begin{equation}\label{koszulcomplex deg 0}
\begin{tikzcd} 
\cdots \arrow{r}{\delta_K} & \bigwedge^2 \g \otimes C^{\infty}(M) \arrow{r}{\delta_K} & \g \otimes C^{\infty}(M)  \arrow{r}{\delta_K} & C^{\infty}(M) \arrow{r} & 0,
\end{tikzcd}
\end{equation}
where $\delta_K : \g \otimes C^{\infty}(M) \rightarrow C^{\infty}(M)$ is defined as $\delta_K(u \otimes 1) = J^*u$ and $J(1 \otimes f) = 0,$ for $u \in \g$ and $f \in C^{\infty}(M),$ and extended to $\bigwedge^\bullet \g \otimes C^\infty(M)$ as an odd derivation. 
Whenever $0 \in \g^*$ is a regular value of the moment map $J: M \rightarrow \g^*,$ (\ref{koszulcomplex deg 0}) provides a resolution of the quotient $C^{\infty}(M)/I \cong C^{\infty}(J^{-1}(0)),$ where $I$ denotes the vanishing ideal of $J ^{-1}(0).$ On the other hand, to take into account gauge invariance, we consider the Chevalley-Eilenberg complex 
 
\begin{equation}\label{CE eilenberg deg 0}
\begin{tikzcd}
0 \arrow{r} & \g \otimes C^{\infty}(M) \arrow{r}{\delta_{CE}} &  \g^* \otimes \g \otimes C^{\infty}(M)  \arrow{r}{ \delta_{CE}} & \cdots,
\end{tikzcd}
\end{equation}
since $\bigwedge^\bullet \g \otimes C^{\infty}(M)$  is $\g$-module. The BFV algebra is then constructed by bringing together the complexes (\ref{koszulcomplex deg 0}) and (\ref{CE eilenberg deg 0}). Notice that the differentials  $\delta_K$ and $\delta_{CE}$ are recovered from the inner derivation $\{Q_{\g}, \, \cdot \, \}.$

\begin{remark}\label{remark: BFV in general}
    The BFV construction applies, more generally, to a first class ideal, that is, an ideal $I \subset C^\infty(M)$ satisfying $\{I, I\} \subset I,$ but with this condition not necessarily given in terms of the structural constants of a Lie algebra $\g.$  If the ideal $I$ is regular, that is, if its vanishing set on $M$ is a embedded submanifold, the Koszul complex still provides a resolution of the quotient $C^\infty(M)/I;$ however, the construction of the corresponding BRST charge requires the application of homological perturbation methods (\cite{Stasheff, Henneaux's Book}). For non-regular ideals, the Koszul complex is not acyclic in general, so one needs first to replace it by the Koszul-Tate resolution, which leads to an enlargement of the ghost/antighost spectrum by the addition of ghosts of ghosts and so on, and then apply the homological perturbation theory to obtain the BRST charge (\cite{Stasheffetal, Stasheff, Kimura, Henneaux's Book}).
\end{remark}

Since the BRST differential is an inner derivation, the cohomology $ H^\bullet_{\{Q_{\g}, \, \cdot \, \}}(\cK)$ naturally inherits a super Poisson structure from $\mathcal{K}^\bullet.$  In particular, the cohomology in degree zero $H^0_{\{Q_{\g}, \, \cdot \, \}}(\cK)$  carries the structure of a Poisson algebra. Then, from the acyclicity of the Koszul complex, one concludes that the natural map 
\begin{align}
H^0_{\{Q_{\g}, \, \cdot \, \}}(\cK) & \longrightarrow \frac{N(I)}{I} \label{Map 0th cohomology and N(I)/I} \\
[(x^{0,0}, x^{1,1}, \cdots)] & \longmapsto \overline{x^{0, 0}}. \nonumber
\end{align}
is a Poisson algebra isomorphism. Under the hypothesis of the classical Marsden-Weinstein theorem, the reduced space $C_{red} \coloneqq J^{-1}(0)/G$ is a Poisson manifold and, in this case, we have  $N(I)/I \cong C^{\infty}(C_{red})$ as Poisson algebras. This the sense in which  we say that the BFV  procedure provides a homological description of hamiltonian reduction.

In graded-geometric terms, the BFV  construction is interpreted as follows: the complex $\cK^\bullet$  can be seen as the algebra of functions on the $\mathbb{Z}$-graded manifold $ M \times T^{*}\g^*[-1],$ where the BRST charge is a function of degree one and the BRST differential $\{Q_\g, \cdot \}$ is a homological vector field. From this perspective, in the ``extended phase space'' $\cM,$ the single  canonical symmetry $\{Q_{\g}, \, \cdot \, \}$ encodes all the infinitesimal symmetries given by the $\g$-action on the ``standard phase space'' $M.$


\section{BFV in degree one}\label{section: brst in degree one}

In this section, we construct the BFV model for the hamiltonian reduction of the shifted cotangent bundle $\cM \coloneqq T^*[1]M.$   
\vspace{0.3cm}

Given a moment map $J^\sharp: \ov\g[-1] \rightarrow C^{\infty}(\cM)$ describing the hamiltonian action of a graded Lie algebra $\ov\g = \h[1]\oplus\g,$  we will provide a homological description of the algebra of functions of the corresponding reduced space. For this, we follow the geometric perspective presented in the end of the last section: we look for a graded manifold endowed with a (hamiltonian) homological vector field, out of which we shall be able to construct a complex that encodes the hamiltonian reduction. 

\subsection{The underlying algebra of the degree one BFV model} 
In analogy with the degree zero case, consider the product 
$$
\cN:= \cM \times T^*[1]\ov\g^*[-1],
$$
in such a way that
\begin{equation}
C^\infty(\cN) = C^{\infty}(\cM)  \otimes C^\infty (T^*[1]\ov\g^*[-1])  =  C^{\infty}(\cM) \otimes S^{\b}(\ov\g^*[-1]) \otimes  S^{\b}(\ov\g).
\end{equation} 

Note that $\cN$ has a natural degree $-1$ Poisson structure. Indeed, the canonical paring between $\ov\g$ and $\ov\g^*,$ which is of degree $-1$ in $T^*[1]\,\ov\g^*[-1],$ and the Schouten bracket in $\cM,$ also of degree $-1,$ can be extended through the Leibniz rule to a Poisson structure on $C^{\infty}(\cN).$\par 

Let 
\begin{equation}
A^{p, q} \coloneqq  C^{\infty}(\cM) \otimes  S^{p}(\ov\g^*[-1]) \otimes S^{q}(\ov\g) \subseteq C^\infty(\cN).
\end{equation}
As in the case of degree zero, $a^{p, q} \in A^{p, q}$ will be said to have {\em ghost number} $p,$ {\em antighost number} $q,$ and {\em total ghost number} $p-q.$ However, notice that, in this case, the grading of $C^{\infty}(\cN)$ induced by total ghost number does not coincide with the $\bZ$-grading that it already carries. So the $\bZ$-graded algebra $C^{\infty}(\cN)$ becomes bigraded, and we let $\cA^{k, \ell}$ denote the subspace of $C^{\infty}(\cN)$ consisting of functions that have total ghost number $k$ and function degree $\ell.$ Then we take the underlying algebra of the BFV model to be $\cA \coloneqq C^{\infty}(\cN)$ with the additional grading by total ghost number.

 \subsection{The BRST charge} The BRST charge should be a function in $\cA \coloneqq C^\infty(\cN)$ whose corresponding inner derivation recovers  the constraints given by the moment map $J^{\sharp}: \ov\g[-1] \rightarrow C^{\infty}(\cM)$ and encodes gauge invariance.
   
We start by considering 
$$Q_0 = J^{\sharp}_1(u^i) 
u^*_i + J^{\sharp}_0(v^j)  v_j^*,$$
where $\{u^i\}$ and $\{v^j\}$ are bases for $\g$ and $\h[1],$ respectively, whereas $\{u^*_i\}$ and $\{v^*_j\}$ denote the corresponding dual bases. 
Since $\{u_i^*, u^j\} = \delta^j_{i},$ for $u^i \in \g$ and $u^*_j \in \g^*$,  and similarly for the pairing between $\h$ and $\h^*,$ it follows that $\{Q_0, u^i\} =  J^{\sharp}_1(u^i)$ and $\{Q_0, v^j\} = J^{\sharp}_0(v^j).$\par 
Next, we notice that the Poisson structure on $\cN$ satisfies  
\begin{equation}\label{relation of the bigrading under de dg -1 Poisson structure}
\{ A^{i, j}, A^{k, \ell} \} \subseteq A^{i + k, j + \ell} \oplus A^{i + k - 1 , j + \ell - 1}.
\end{equation}
Thus, as $Q_0 \in A^{1, 0},$ we see that
\begin{equation}
\{\, Q_0, \, A^{p, q}\, \} \subseteq A^{p, q - 1} \oplus A^{p+1, q}.
\end{equation}
Hence, considering the action of $\{Q_0, \, \cdot\, \}$ restricted to $A^{0, q},$ $q \in \bN,$ and taking the projection onto $A^{0, q-1},$ we obtain the complex

\begin{equation}\label{Koszul complex in deg 1}
\begin{tikzcd} 
\cdots \arrow{r}{\delta} &  C^{\infty}(\cM) \otimes S^2 (\ov\g) \arrow{r}{\delta} &  C^{\infty}(\cM) \otimes \ov\g  \arrow{r}{\delta} & C^{\infty}(\cM) \arrow{r} & 0,
\end{tikzcd}
\end{equation}
which is a graded version of the Koszul complex (\ref{koszulcomplex deg 0}). As in the degree zero case, the Koszul complex (\ref{Koszul complex in deg 1}) is, under appropriate conditions, a resolution of the quotient $C^{\infty}(\cM)/\cJ,$ where $ \cJ$ is the sheaf of ideals spanned by the constraints given by the moment map. More precisely, we have the following key result.

\begin{proposition}\label{acyclicity of koszul complex in degree one}
If $0 \in \h^*$ is a regular value of the moment map $(J, J^{\sharp}): \cM \rightarrow (\ov\g[-1])^*,$ then the Koszul complex (\ref{Koszul complex in deg 1}) is acyclic.
\end{proposition}
\begin{proof}
It follows from a simple adaptation of the proof of Theorem 9.1 in \cite{Henneaux's Book}.
\end{proof}

Proposition \ref{acyclicity of koszul complex in degree one} shows that $Q_0$ homologically describes the procedure of passing to the constraint submanifold $\cC \coloneqq (J, J^{\sharp})^{-1}(0)$ associated to the sheal of ideals $\cJ.$ On the other hand, for $P \in A^{0, 0} = C^{\infty}(\cM),$   $\{Q_0, P\} \in A^{1, 0}$ is given by 
$$
\{Q_0, P\} = (-1)^{|P| - 1}\{J^{\sharp}_1(u^i), P\}u^*_i + \{J^{\sharp}_0 (v^j), P\}v^*_j,
$$ 
which is zero if and only if $\{J^{\sharp}_1(u^i), P\} = \{J^{\sharp}_0 (v^j), P\} = 0;$ hence, since $J^{\sharp}: \ov\g [-1] \rightarrow C^{\infty}(\cM)$ is a moment map for the infinitesimal action of $\ov\g,$ the inner derivation $\{Q_0, \, \cdot \, \}$ also captures gauge invariance. 

Notice, however, that $Q_0$ does not necessarily satisfy $\{Q_0, Q_0\} = 0.$ Starting from $Q_0$ and trying to implement this condition,   we arrive at the following function of degree two  and total ghost number one,
\begin{equation}\label{BRST charge for degree one hamiltonian reduction}
Q_{\ov\g} = J^{\sharp}_1(u^i) u^*_i + J^{\sharp}_0(v^j)  v_j^* - \frac{1}{2} c_k^{ij} \, u^*_i \, u^*_j \, u^k - d^{mn}_{p} u^*_{m} \, v^*_{n} \, v^p,
\end{equation}
where $\{c^{ij}_k\}$ and $\{d^{mn}_p\}$ are the structural constants of the graded Lie algebra $\ov\g$ (the former are associated to the basis $\{u_i\}$ of $\g$ and the latter are obtained by writting the action of $\g$ on $\h[1]$ with respect of to the basis $\{v_m\}$ of $\h[1]$), which, due to the Jacobi identity for the graded Lie algebra 
$\ov\g,$ satisfies $\{Q_{\ov\g}, Q_{\ov\g}\} = 0.$ Then, from  the graded Jacobi identity for $\{\, \cdot \, , \, \cdot \, \},$ one concludes that $\{Q_{\ov\g}, \{Q_{\ov\g},\,  \cdot \, \}\}$ $= 0,$ that is, $\{Q_{\ov\g},\,  \cdot \, \}$ is a homological vector field on $\cN. $ Therefore, we consider the differential graded algebra $(\cA, \{Q_{\ov\g}, \, \cdot \, \}).$ \par  

Observe that, since the constraints we are considering are regular and come from a graded Lie algebra action, we neither need to resort to the Koszul-Tate resolution to obtain a resolution of $C^\infty(\cM)/\cJ$ nor to homological perturbation methods to construct the BRST charge (see Remark \ref{remark: BFV in general}). So the BFV procedure for the hamiltonian reduction of $\cM  \coloneqq T^*[1]M$ is really the graded analog of what was considered by Kostant-Sternberg in \cite{Kostant}. The case of regular irreducible first class constraints on $\cM$ was treated in \cite{Sharapov}.\par 

\subsection{BRST cohomology at degree zero} Next, we show that the differential graded algebra $(\cA, $ $\{Q_{\ov\g}, \cdot\})$ is indeed the BFV algebra for the degree one reduction of $\cM \coloneqq T^*[1]M$ by the hamiltonian action of the graded Lie algebra $\ov\g.$ More precisely, we will prove the following theorem, which is just a particular case of a more general statement found in \cite{AleZabzine} (c.f. Prop.~10 therein).

\begin{theorem}\label{main result}
The cohomology of the complex $(\cA, \{Q_{\ov\g}, \, \cdot \, \})$ at total ghost number zero is so that the natural map
\begin{align}
\Phi : H^{0, \bullet}_{\{Q_{\ov\g}, \, \cdot \, \}}(\cA) & \longrightarrow \frac{N(\cJ)}{\cJ} \label{Map 0th cohomology and N(I)/I in the dg one case} \\
[(x^{0,0}, x^{1,1}, \cdots)] & \longmapsto \overline{x^{0, 0}} \nonumber
\end{align} 
is an isomorphism of degree $-1$ Poisson algebras.
\end{theorem}
 Under the hypothesis of the Cattaneo-Zambon reduction theorem for $\cM \coloneqq T^{*}[1]M$ (Theorem \ref{marsden-weinstein cattaneozambon}), the above theorem provides the degree one version of the classical result by Kostant-Sternberg on homological reduction in the usual (nongraded) hamiltonian context (\cite{Kostant}). 

From (\ref{relation of the bigrading under de dg -1 Poisson structure}) and the fact that terms in $A^{2, 1}$ appearing in $Q_{\ov\g}$ have constant coefficients, it follows that 
\begin{equation}
\{\, Q_{\ov\g}, \, A^{p, q}\, \} \subseteq A^{p, q - 1} \oplus A^{p+1, q} .
\end{equation} 
Therefore, the action of $\{Q_{\ov\g},\,  \cdot \, \}$ on $\cA$ splits as shown by  the following diagram:
\begin{equation}\label{diagram: splitting of Q}
\begin{tikzcd}
A^{0, 0} \arrow{r}{\delta_H} & A^{1, 0} \arrow{r}{\delta_H} & A^{2, 0} \arrow{r}{\delta_H} & \cdots \\
A^{0, 1} \arrow{u}{\delta_V} \arrow{r}{\delta_H} & A^{1, 1} \arrow{r}{\delta_H} \arrow{u}{\delta_V} & A^{2, 1} \arrow{r}{\delta_H} \arrow{u}{\delta_V} & \cdots \\
A^{0, 2} \arrow{r}{\delta_H} \arrow{u}{\delta_V} & A^{1, 2} \arrow{u}{\delta_V} \arrow{r}{\delta_H} & A^{2, 2} \arrow{r}{\delta_H} \arrow{u}{\delta_V}& \cdots  \\ 
\vdots \arrow{u} & \vdots \arrow{u} & \vdots \arrow{u}
\end{tikzcd}
\end{equation}
In other words, $\{Q_{\ov\g}, \, \cdot \, \}$ can be written as 
\begin{equation}\label{splitting of the brst differential in two derivations}
\{Q_{\ov\g}, \, \cdot \, \} = \delta_H + \delta_V.
\end{equation}
The condition  that $\{Q_{\ov\g},\,  \cdot \, \}$ squares to zero implies that the maps $\delta_V$ and $\delta_H$ are anticommuting differentials. Moreover, the first column of the diagram (\ref{diagram: splitting of Q}) gives precisely the complex (\ref{Koszul complex in deg 1}), which is the graded version of the Koszul complex (\ref{koszulcomplex deg 0}), whereas the other columns are obtained from the first one by letting $\delta_V$ act as the identity on the parts coming from $\ov\g^*.$ So the differential $\delta_V$ is simply the Koszul differential; in particular, by Proposition \ref{acyclicity of koszul complex in degree one}, it follows that $H^k_{\delta_V} = 0,$ for $k \neq 0.$ On the other hand, the differential $\delta_H$ can be seen as the graded version of the Chevalley-Eilenberg differential.\par 

Notice that the grading of $\cA \coloneqq C^{\infty}(\cN)$ by function degree and its Poisson structure naturally descend to the cohomology of $(\cA, \{Q_{\ov\g},\,  \cdot \, \}),$ so we have the following. 

\begin{lemma}
The cohomology $H^\bullet_{\{Q_{\ov\g}, \, \cdot \, \} }(\cA),$ which is computed with respect to the total ghost number, is naturally graded by the function degree. More precisely, for all $p \in \bZ,$ we have
\begin{equation}\label{brst cohomology graded by function degree}
H^{p, \bullet}_{\{Q_{\ov\g}, \, \cdot \, \}}(\cA) \coloneqq H^{p}_{\{Q_{\ov\g}, \, \cdot \, \}}(\cA) = \bigoplus_{ \substack{k \geq p \\ k \in \bZ}} H_{\{Q_{\ov\g}, \, \cdot \, \}}^{p, k}(\cA),
\end{equation}
where $H_{\{Q_{\ov\g}, \, \cdot \, \}}^{p, k}(\cA)$ denotes the space of all cohomology classes of function degree $k$ at total ghost number $p.$ Moreover, the Poisson structure on $\cA$ descends to the graded algebra $H^{0, \bullet}_{\{Q_{\ov\g}, \, \cdot \, \}}(\cA).$
\end{lemma}
\begin{proof}
The decomposition (\ref{brst cohomology graded by function degree}) follows from the fact that both $\delta_H$ and $\delta_{V}$ are derivations of function degree one. To prove that $H^{0, \bullet}_{\{Q_{\ov\g}, \, \cdot \, \}}(\cA)$ is a Poisson algebra, one can directly verify that  $ker({ \{Q_{\ov\g}, \,  \cdot \,\}}: \cA^0 \rightarrow \cA^1 )$  is a Poisson subalgebra of $\cA,$ whereas $im({ \{Q_{\ov\g}, \,  \cdot \,\}}: \cA^{-1} \rightarrow \cA^0) \subset$ is a Poisson ideal of this kernel. \end{proof}

Given all this, the proof of Theorem \ref{main result} consists of a standard argument, but we provide it here for the sake of completeness. It will be divided into the following four steps.\par 

\vspace{0.3cm} 
\textit{Step 1:} Firstly, we claim that the map $\Phi : H^{0, \bullet}_{\{Q_{\ov\g}, \, \cdot \, \}}(\cA) \rightarrow N(\cJ)/\cJ$ is well defined. To verify that, let $[a] \in H^{0, \bullet}_{\{Q_{\ov\g}, \, \cdot \, \}}(\cA),$ where  $a = \sum_{i \in \bN} a^{i, i}\in \cA^0$, and observe that 
\begin{equation}\label{bracker of Q and a00}
\{Q_{\ov\g}, a^{0,0}\} = (-1)^{(|a^{0, 0}| -1)} \{J^{\sharp}_1 u^i, a^{0, 0}\}u^*_i + \{J^{\sharp}_0 v^j, a^{0, 0}\}v^*_j \in A^{1, 0}.
\end{equation}
From the diagram (\ref{diagram: splitting of Q}), one can see that among all $\{Q_{\ov\g}, a^{i,i}\},$ $i \in \mathbb{N},$ the only one that has a component in $A^{1, 0}$ is $\{Q_{\ov\g}, a^{1,1}\}$. Thus, since $\{Q_{\ov\g}, a\} = 0,$ this component cancels out with $\{Q_{\ov\g}, a^{0, 0}\}$.  By writting a generic expression for $a^{1, 1} \in A^{1, 1},$ 
$$ \alpha^{i_1}_{j_1} u^*_{i_1} u^{j_1} + \beta^{i_2}_{j_2} u^*_{i_2} v^{j_2} + \gamma^{i_3}_{j_3} v^*_{i_3} u^{j_3} + \theta^{i_4}_{j_4} v^*_{i_4} u^{j_4},$$ 
where coefficients are in $C^{\infty}(\cM),$
using this to compute $\{Q, a^{1, 1}\}$ and comparing the $A^{1, 0}$-components of the result with (\ref{bracker of Q and a00}), one concludes that $a^{0, 0} \in N(\cJ).$  On the other hand, for $[a], [b] \in H^{0, \bullet}_{\{Q_{\ov\g}, \, \cdot \, \}}(\cA)$ such that $[a] = [b],$ we have 
$$a - b \in im(\{Q_{\ov\g},\, \cdot \, \}: \cA^{-1} \rightarrow \cA^{0}),$$  
which, in particular, implies that $a^{0, 0} - b^{0, 0} \in \cJ.$ Therefore, we conclude that the map $\Phi : H^{0, \bullet}_{\{Q_{\ov\g}, \, \cdot \, \}}(\cA) \rightarrow N(\cJ)/\cJ$ is well defined.

\textit{Step 2 (Injectivity)}: Let  $[a], [b] \in H^{0, \bullet}_{\{Q_{\ov\g}, \, \cdot \, \}}(\cA)$ be such that $a^{0, 0} = b^{0, 0}$ in $N(\cJ)/\cJ,$ that is, $a^{0, 0} - b^{0, 0} \in \cJ.$ Then, there exists $c^{0, 1} \in A^{0, 1}$ such that $\delta_V (c^{0, 1}) = a^{0, 0} - b^{0, 0}. $ Since $a, b \in \cA$ are cocycles, we have 
$$
\begin{cases}
\delta_H(a^{0, 0}) + \delta_V(a^{1, 1}) & = 0\\
\delta_H(b^{0, 0}) + \delta_V(b^{1, 1}) & = 0;
\end{cases}
$$
from this, we obtain that
$$
\delta_V(a^{1, 1} - b^{1, 1}) = - \delta_H(a^{0, 0} - b^{0, 0}) = \delta_H \delta_V(c^{0, 1}) = \delta_V\delta_H(c^{0, 1}),  
$$
that is, $\delta_V(a^{1, 1} - b^{1, 1} - \delta_H(c^{1, 0})) = 0,$ which shows that $a^{1, 1} - b^{1, 1} - \delta_H(c^{1, 0}) \in A^{1, 1}$ is a $\delta_V$-cocycle. By the acyclicity of the Koszul complex, we obtain that there exists $c^{1, 2} \in A^{1, 2}$ such that $$\delta_V(c^{1, 2}) = a^{1, 1} - b^{1, 1} - \delta_H(c^{1, 0}).$$ Now, notice that $c^{0, 1} + c^{1, 2} \in \cA^{-1}$ satisfies 
\begin{align*}
\{Q_{\ov\g}, c^{0, 1} + c^{1, 2}\} & = (\delta_V + \delta_H)(c^{0, 1} + c^{1, 2})\\
& =  (a^{0, 0} - b^{0, 0}) + (b^{0, 0} - b^{1, 1}) + \delta_H(c^{1, 2}).
\end{align*}

We conclude by induction: observe that the equations

\begin{equation}\label{eq for a and b cocycles}
\begin{cases}
\delta_H(a^{n, n}) + \delta_V(a^{n+1, n+1}) &= 0\\
\delta_H(b^{n, n}) + \delta_V(b^{n+1, n+1}) &= 0;
\end{cases}
\end{equation}
 hold for all $n \in \mathbb{N}$, since $a, b \in \cA^0$ are cocycles. For a fixed $n \in \mathbb{N},$ assume that \begin{equation}\label{induction hyp for the inject.}
\delta_V(c^{n, n+1}) + \delta_H(c^{n-1, n}) = a^{n, n} - b^{n, n}.
\end{equation}
Then, using (\ref{eq for a and b cocycles}) and (\ref{induction hyp for the inject.}), we obtain  
\begin{align*}
\delta_V(a^{n+1, n+1} - b^{n+1, n+1}) & =  - \delta_H(a^{n, n} - b^{n, n}) \\
& = - \delta_H(\delta_V(c^{n, n+1}) + \delta_H(c^{n-1, n})) \\
& =  \delta_V\delta_H(c^{n, n+1}), 
\end{align*}
that is,  $$\delta_V(a^{n+1, n+1} - b^{n+1, n+1} -  \delta_H(c^{n, n+1}) = 0,$$ which shows that $a^{n+1, n+1} - b^{n+1, n+1} -  \delta_H(c^{n, n+1}) \in A^{n+1, n+1}$ is a $\delta_V$-cocycle. Since the Koszul complex is acyclic, we conclude that there exists $c^{n+1, n+2} \in A^{n+1, n+2}$ such that $$\delta_V(c^{n+1, n+2}) + \delta_H(c^{n, n+1}) = a^{n+1, n+1} - b^{n+1, n+1},$$
which is just (\ref{induction hyp for the inject.}) for $n+1.$ Hence, the induction produces $c = \sum_{k \in \mathbb{N}} c^{k, k+1} \in \cA^{-1}$ such that $\{Q_{\ov\g}, \, c\} = a - b,$ which means exactly that $[a] = [b]$ in $H^{0, \bullet}_{\{Q_{\ov\g}, \, \cdot \, \}}(\cA).$ In this way, we conclude that the map $\Phi : H^{0, \bullet}_{\{Q_{\ov\g}, \, \cdot \, \}}(\cA) \rightarrow N(\cJ)/\cJ$ is indeed injective.

\textit{Step 3 (Surjectivity) }: Let $f \in N(\cJ)$ be a homogeneous multivector field. We will construct $a = f + \sum_{k \geq 1} a^{k, k}$ satisfying $\{Q_{\ov\g}, \, a\} = 0.$ Firstly, notice that 
$$\{Q_{\ov\g}, f\} = \delta_H(f) = \{J^{\sharp}_0(v^j), f\}v^*_j + (-1)^{(|f| - 1)} \{J^{\sharp}_1(u^i), f\}u^*_i \in A^{1, 0}.$$
Since $f \in N(\cJ),$ it follows that there exists $a^{1, 1} \in A^{1, 1}$ such that $\delta_V(a^{1, 1}) = - \delta_H(f).$ Hence, we see that $$\{Q_{\ov\g}, f + a^{1, 1}\} = \delta_H(f) + \delta_V(a^{1, 1}) + \delta_H(a^{1, 1}) = \delta_H (a^{1, 1}).$$ 
By the anticommutative of $\delta_V$ and $\delta_H,$ we obtain $$ \delta_V\delta_H (a^{1, 1}) = - \delta_H\delta_V(a^{1, 1}) = \delta_H \delta_H(f) = 0,$$ which says that $\delta_H (a^{1, 1}) \in A^{2, 1}$ is a $\delta_V$-cocycle. Thus, the acyclicity of the Koszul complex implies that there exists $a^{2, 2} \in A^{2, 2}$ satisfying $\delta_V(a^{2, 2}) = - \delta_H(a^{1, 1}).$ Now, for $f + a^{1, 1} + a^{2, 2},$ we obtain $$\{Q_{\ov\g}, f + a^{1, 1} + a^{2, 2}\} = \delta_H(a^{2, 2}).
$$ 
The induction step is then clear. Hence, we obtain $a = f + \sum_{k \geq 1} a^{k, k} \in \cA^ 0$ satisfying $\{Q_{\ov\g}, a\} = 0.$ Since the function degree of $a^{k, k} \in A^{k, k},$ for all $k \in \bN,$ is necessarily equal to the function degree of the homogeneous multivector field $f \in N(\cJ),$ it follows that $a^{k, k} \neq 0$ only for finitely many $k \in \bN.$  Then $a = f + \sum_{k \geq 1} a^{k, k}$ defines a cohomology class $[a] \in  H^{0, \bullet}_{ \{Q_{\ov\g}, \,  \cdot \,\} }(\cA)$ such that $\Phi([a]) = f,$ which shows that $\Phi : H^{0, \bullet}_{\{Q_{\ov\g}, \, \cdot \, \}}(\cA) \rightarrow N(\cJ)/\cJ$ is surjective.\par  
\textit{Step 4:} Finally, we claim that  $\Phi : H^{0, \bullet}_{\{Q_{\ov\g}, \, \cdot \, \}}(\cA) \rightarrow N(\cJ)/\cJ$ is a Poisson algebra morphism. For $a, b \in \cA^0,$ where $a = \sum_{i \in \bN}a^{i, i}$ and $b = \sum_{j \in \bN} b^{j, j},$ we have $$\{a, b\} \coloneqq \sum_{i, j \in \bN} \{a^{i, i}, b^{j, j}\};$$ however, from (\ref{relation of the bigrading under de dg -1 Poisson structure}), it is easy to see that among all the brackets $ \{a^{i, i}, b^{j, j}\},$ $i, j \in \bN,$ the only one that has a component in  $A^{0, 0}$ is $\{a^{0, 0}, b^{0, 0} \}.$ Hence, for $[a], [b] \in H^{0, \bullet}_{ \{Q_{\ov\g}, \,  \cdot \,\} }(\cA),$ we obtain that $$\Phi(\{[a], [b]\}) \coloneqq (\{a, b\})^{0, 0} = \{a^{0, 0}, b^{0, 0}\} = \{ \Phi([a]), \Phi([b])\},$$
which shows that the map $\Phi : H^{0, \bullet}_{\{Q_{\ov\g}, \, \cdot \, \}}(\cA) \rightarrow N(\cJ)/\cJ$ is a Poisson algebra morphism. \par 

All together, the above four steps prove that the map $\Phi : H^{0, \bullet}_{\{Q_{\ov\g}, \, \cdot \, \}}(\cA) \rightarrow N(\cJ)/\cJ$ is indeed an isomorphism of degree $-1$ Poisson algebras, which concludes the proof of Theorem \ref{main result}. 

\section{The homotopy structure of Poisson reduction}\label{section: the homotopy nature of Poisson reduction}

In this section, we show how Poisson structures can be incorporated into the degree one BFV construction. This leads to the homotopy structure behind  Poisson reduction and provides the homological version of the generalized Marsden-Weinstein theorem (Theorem \ref{generalized context - Marsden Weinstein thm}).  
\vspace{0.2cm}

\subsection{A derived bracket construction}\label{subsection: derived bracket}
If the pair $(\psi \coloneqq J^{\sharp}|_{\g}, J)$ associated to the moment map $J^{\sharp}: \ov\g[-1] \rightarrow C^{\infty}(\cM)$ is regular, 
Theorem \ref{main result} says that we have, for each $k \in \bN,$ the isomorphism of vector spaces
\begin{equation}\label{H0k brst cohomology is deg k multivector fields on the reduced space}
H^{0, k}_{\{Q_{\ov\g},\, \cdot\, \}}(\cA) \cong \mathfrak{X}^k(C_{red}).
\end{equation}
In particular, for $k = 0,$ we obtain that $H^{0, 0}_{\{Q_{\ov\g},\, \cdot\, \}}(\cA)$ is isomorphic to $C^{\infty}(C_{red})$ as an algebra.\par
On the other hand, a weak Poisson structure $\pi \in \mathfrak{X}^2(M)$ descends to a Poisson bivector  $\pi_{red} \in \mathfrak{X}^2(C_{red})$ if and only if $\pi \in N(\cJ),$ where $\cJ \subset C^{\infty}(\cM)$ is the homogeneous ideal generated by the constraints given by the moment map  $J^{\sharp}: \ov\g[-1] \rightarrow C^{\infty}(\cM).$ The identification (\ref{H0k brst cohomology is deg k multivector fields on the reduced space}), for $k = 2,$ implies the following result.

\begin{corollary}\label{cohomological condition for the reducibility of pi}
Let $\pi \in  C_2^{\infty}(\cM) = \mathfrak{X}^2(M)$ be a weak Poisson structure on $M.$ Then $\pi$ is reducible with respect to the submanifold $C \coloneqq J^{-1}(0)$ and the integrable distribution $\cD = \langle J_1^{\sharp}(u) \rangle_{u \in \g_0}$ if and only if there exists a cohomology class in $[\Pi] \in H^{0, 2}_{\{Q_{\ov\g}, \, \cdot \, \}}(\cA)$ for which $\Pi^{0, 0} = \pi,$ that is, $\Phi ([\Pi]) = 
\overline{\pi} \in N(\cJ)/\cJ,$ for $\Phi : H^{0, \bullet}_{\{Q_{\ov\g}, \, \cdot \, \}}(\cA) \rightarrow N(\cJ)/\cJ$ the canonical map in (\ref{Map 0th cohomology and N(I)/I in the dg one case}). 
\end{corollary}

Now, fix $\pi \in \mathfrak{X}^2(M)$ a reducible weak Poisson structure on $M,$ and  let $[\Pi] \in H^{0, 2}_{\{Q_{\ov\g}, \, \cdot \, \}}(\cA)$ be the corresponding cohomology class. For $[a], [b] \in H_{\{Q_{\ov\g}, \, \cdot \, \}}^{0, 0}(\cA),$ we define the derived bracket 

\begin{equation}\label{derived poisson bracket on H00}
\{[a], [b]\}_{\Pi} \coloneqq \{\{[\Pi], [a]\}, [b]\} = [\{\{ \Pi, a\}, b \}].
\end{equation}

We claim that this is a Poisson bracket on $H^{0, 0}_{\{Q_{\ov\g},\, \cdot \, \}}(\cA).$ Indeed, observe that the skew-symmetry and Leibniz rule for $\{\cdot, \cdot\}_{\Pi}$ are simple consequences of the same properties for the Poisson bracket on $\cA.$ As for the Jacobi identity, it is enough to notice that
\begin{equation}\label{00 component of (Pi, Pi)}
\{\Pi, \Pi\}^{0, 0} = \{\pi, \pi\} \in \cJ
\end{equation}
so that $\{[\Pi], [\Pi]\} = [\{\Pi, \Pi\}] = 0 \in H^{0, 3}_{\{Q_{\ov\g},\, \cdot \, \}} (\cA).$ 
It follows that $(H^{0, 0}_{\{Q_{\ov\g}, \cdot\}}(\cA), \{\cdot, \cdot\}_{\Pi})$ is a Poisson algebra. \par 

Similarly, the reducible weak Poisson structure $\pi \in \mathfrak{X}^2(M)$ induces a Poisson bracket on the degree zero functions in $N(\cJ)/\cJ,$ that is, for $\overline{f}, \overline{g} \in (N(\cJ)/\cJ)_0,$ 
\begin{equation}\label{derived bracket on deg function in the normalizer}
\{\overline{f}, \overline{g} \}_{\pi} \coloneqq \overline{ \{\{\pi, f\}, g \}}
\end{equation} 
is a Poisson bracket; and, for this, 
 we have 
\begin{equation}\label{isomorphism: reduced poisson algebra and normalizer with derived bracket}
\bigg( \bigg(\frac{N(\cJ)}{\cJ}\bigg)_0, \{\cdot, \cdot\}_{\pi} \bigg) \cong (C^{\infty}(C_{red}), \{\cdot, \cdot\}_{\pi_{red}}).
\end{equation}

As the next result shows, the BRST cohomology at total ghost number zero and function degree zero, with the derived bracket (\ref{derived poisson bracket on H00}), recovers the left-hand side of (\ref{isomorphism: reduced poisson algebra and normalizer with derived bracket}).

\begin{proposition}
Let $\pi \in \mathfrak{X}^2(M)$ be a reducible weak Poisson structure, and let $[\Pi] \in H^{0, 0}_{\{Q_{\ov\g}, \, \cdot \, \}}(\cA)$ be the cohomology class for which $\Pi^{0, 0} = \pi.$ If we endow $(N(\cJ)/\cJ)_0$ and $H^{0, 0}_{\{Q_{\ov\g}, \, \cdot \, \}}(\cA)$ with the  Poisson brackets (\ref{derived bracket on deg function in the normalizer}) and (\ref{derived poisson bracket on H00}), respectively, then the natural map
\begin{align}
\Phi : H^{0, 0}_{\{Q_{\ov\g}, \, \cdot \, \}}(\cA) & \longrightarrow \bigg(\frac{N(\cJ)}{\cJ}\bigg)_0  \label{isomorphim: H00 with derived bracket and N(j)/j} \\
[(a^{0,0}, a^{1,1}, \cdots)] & \longmapsto \overline{a^{0, 0}} \nonumber
\end{align}
is an isomorphism of Poisson algebras.
\end{proposition}

\begin{proof}
From Theorem \ref{main result}, we already know that the map in (\ref{isomorphim: H00 with derived bracket and N(j)/j}) is an isomorphism of algebras. We need to prove that it is a Poisson morphism. For $[a], [b] \in H^{0, 0}_{\{Q_{\ov\g}, \, \cdot \, \}}(\cA),$ notice that  
$$ \{\{\Pi, a\}, b \}^{0, 0} = \{\{\pi, a^{0, 0}\}, b^{0, 0}\}.$$ 
Then, we have
\begin{align*}
\Phi(\{[a], [b]\}_{\Pi}) & = \Phi([ \{ \{\Pi, a\} , b\} ]) \\
& = \overline{\{\{\Pi, a\}, b\}^{0, 0}}\\
& = \overline{\{\pi, a^{0, 0}\}, b^{0, 0} \}}\\
& = \{\overline{a^{0, 0}}, \overline{b^{0, 0}}\}_{\pi}\\
& = \{\Phi([a]), \Phi([b])\}_{\pi}.
\end{align*}
\end{proof}

In view of (\ref{isomorphism: reduced poisson algebra and normalizer with derived bracket}), the last proposition implies that 
\begin{equation}\label{isomorphism: H00 and the reduced Poisson algebra}
(H^{0, 0}_{\{Q_{\ov\g,\,  \cdot \, }\}}(\cA), \{\cdot, \cdot\}_{\Pi}) \cong (C^{\infty}(C_{red}), \{\cdot, \cdot\}_{\pi_{red}}).
\end{equation}
We conclude that, given a reducible weak Poisson structure $\pi \in \mathfrak{X}^2(M)$, the degree one BFV construction recovers, through the derived bracket construction (\ref{derived poisson bracket on H00}), the Poisson algebra structure on functions on the reduced space $C_{red}.$ \par  

Furthermore,  
since $\{[\Pi], [\Pi]\} = [\{\Pi, \Pi\}] = 0 $,
it follows that 
$$
\{[\Pi], \,  \cdot \, \} : H^{0, k}_{\{Q_{\ov\g}, \, \cdot \, \}}(\cA) \rightarrow H^{0, k+1}_{\{Q_{\ov\g}, \, \cdot \, \}}(\cA)
$$ 
is a differential, which shows that  $(H^{0, \bullet}_{\{Q_{\ov\g}, \, \cdot \, \}}(\cA),$ $ \{[\Pi], \, \cdot \, \})$ is a complex. 
The identifications (\ref{H0k brst cohomology is deg k multivector fields on the reduced space}) respect differentials and define an isomorphism of differential graded algebras
$$
(H^{0, \bullet}_{\{Q_{\ov\g}, \, \cdot \, \}}(\cA), \{[\Pi], \, \cdot \, \}) \cong (\mathfrak{X}^\bullet(C_{red}), \{\pi_{red}, \, \cdot \, \}).
$$
Therefore,  
the BRST cohomology $H_{\{Q_{\ov\g}, \cdot \}}^{0, \bullet}(\cA)$  also recovers the Poisson cohomology of the reduced bivector $\pi_{red} \in \mathfrak{X}^2(C_{red}).$
\subsection{Extended BRST charge}\label{subsection: extended brst charge} We now show how to combine the BRST charge $Q_{\ov\g}$, which encodes the reduction data, and the cocycle $\Pi$, which encodes a given reducible weak Poisson structure, into a single {\em extended BRST charge} 
$$
S^{(\infty)} = Q_{\ov\g} + \Pi + \cdots,
$$
which is a formal series of elements in $\cA$ with descending total ghost number. 

\begin{remark}
Our extended BRST charge is directly related to the construction appearing in \cite{Sharapov}.  There the authors start with a BRST charge encoding the reduction data and a reducible weak Poisson structure and, using the acyclicity of the Koszul complex, they extend the BRST charge to a formal power series in the ghost number. In \cite{AleZabzine}, a similar construction is considered, mainly for lifted hamiltonian actions, in the context of general BFV models and with a view towards applications to topological field theories. Here, however, we do something slightly different: we start with the BRST charge $Q_{\ov\g}$ and the cocycle $\Pi$ and use repeatedly the fact that the BRST cohomology vanishes at negative total ghost number (which follows from the acyclicity of the Koszul complex) to construct our extended BRST charge. Note that this may differ from the corresponding charges appearing in \cite{Sharapov} and \cite{AleZabzine} by exact terms.
\end{remark}
 
Notice that
$$
\{Q_{\ov\g} + \Pi, Q_{\ov\g} + \Pi\} = \{\Pi, \Pi\} \in \cA^0.
$$ 
Since $\{\Pi, \Pi\}^{0, 0} = [\pi, \pi] \in \cJ,$ and this implies that $\{\Pi, \Pi\}$ is $\{Q_{\g}, \, \cdot\, \}$-exact,  there exists $\Pi^{(-1)} \in \cA^{-1}$ of function degree two such that $\{Q_{\ov\g,}, \Pi^{(-1)}\} = - \frac{1}{2} \{\Pi, \Pi\}.$ Now let $\Pi^{(0)} \coloneqq \Pi$ and consider $$ S^{(1)} \coloneqq Q_{\ov\g} + \Pi^{(0)} + \Pi^{(-1)}.$$ Observe that $$ \{S^{(1)}, S^{(1)}\} = 2\{\Pi^{(0)}, \Pi^{(-1)}\} + \{\Pi^{(-1)}, \Pi^{(-1)}\} \in \cA^{-1} \oplus \cA^{-2}.$$
From the Jacobi identity, one concludes that $\{\Pi^{(0)}, \Pi^{(-1)}\} \in \cA^{-1}$ is $\{Q_{\ov\g},\, \cdot \, \}$-closed, then it defines a cohomology class in $H_{\{Q_{\ov\g},\, \cdot \, \}}^{-1, \bullet}(\cA);$ since  $H_{\{Q_{\ov\g},\, \cdot \, \}}^{-1, \bullet}(\cA) = 0,$ the cocycle $\{\Pi^{(0)}, \Pi^{(-1)}\}$ is  exact, so there exists $\Pi^{(-2)} \in \cA^{-2}$ for which $\{Q, \Pi^{(-2)}\} = - \{\Pi^{(0)}, \Pi^{(-1)}\}.$ 
Now, for 
$$
S^{(2)} \coloneqq Q_{\ov\g} + \Pi^{(0)} + \Pi^{(-1)} + \Pi^{(-2)}
$$ 
one can check that 
$$
\{S^{(2)}, S^{(2)}\} \in \cA^{-2} \oplus \cA^{-3} \oplus \cA^{-4},
$$ 
and the Jacobi identity implies  that the term in $\cA^{-2}$ is a $\{Q_{\ov\g}, \, \cdot\,  \}$-cocycle, so that  it defines a cohomology class in $H_{\{Q_{\ov\g}, \, \cdot \, \}}^{-2, \bullet}(\cA) = 0,$ implying that this $\cA^{-2}$-term is actually exact.

By induction, we conclude that there exists 
\begin{equation}\label{extended brst charge}
S^{(\infty)} = Q_{\ov\g} + \sum_{k \in \bN} \Pi^{(-k)}
\end{equation}
satisfying  \begin{equation}\label{master equation for the extended brst charge}
\{S^{(\infty)}, S^{(\infty)}\} = 0.
\end{equation}

\subsection{Homotopy structure}\label{subsection: homotopy structure} Let $\cN_{1, -1}$  be the $\bZ$-graded lagrangian submanifold of $\cN \coloneqq \cM \times T^*[1]\ov\g^*[-1]$  whose sheaf of functions is given by $$ C^{\infty}(\cN_{1, -1}) \coloneqq C^{\infty}(M) \otimes \bigwedge^\bullet \g^*  \otimes \bigwedge^\bullet \h.
$$
Denote by $\mathfrak{X}^\bullet(\cN_{1, -1})$ the multiderivations of the algebra $C^\infty(\cN_{1, -1})$, which are thought of as multivector fields on $\cN_{1, -1}.$ The grading in  $\mathfrak{X}^\bullet(\cN_{1, -1})$ is defined by letting
$$\mathfrak{X}^{n}(\cN_{1, -1}) \coloneqq \bigoplus_{p + q = n} \mathfrak{X}^{p, q}(\cN_{1, -1}),$$ 
where $\mathfrak{X}^{p, q}(\cN_{1, -1})$ denote  multivector fields of degree $p$ and polynomial degree $q$. 
The Gerstenhaber bracket  on $\mathfrak{X}^\bullet(\cN_{1, -1})$ is obtained as an extension of the graded commutator of derivations of the $\bZ$-graded algebra $C^{\infty}(\cN_{1, -1})$ (\cite[$\S$4.1]{Relative formality}).

Notice then that we have the following identifications
\begin{equation}\label{identification brst algebra and multivector fields}
 \cA^{k, \ell} \cong \mathfrak{X}^{\ell - k, \, k}(\cN_{1, -1}),
\end{equation}
recalling that the left-hand side denotes elements in $C^{\infty}(\cN)$ of total ghost number $k$ and function degree $\ell.$  From this, we see that the terms $Q_{\ov\g}$ and $\Pi^{(-k)},$ $k \in \bN,$ of the extended BRST charge (\ref{extended brst charge}) are such that  
$$Q_{\ov\g} \in \cA^{1 , 2} \cong \mathfrak{X}^{1, \,1}(\cN_{1, -1}), \, \, \, \Pi^{(-k)} \in \cA^{-k, \,2} \cong \mathfrak{X}^{k+2, -k}(\cN_{1, -1}),$$
so $S^{(\infty)}$ gives a formal bivector on $\cN_{1, -1}.$

The next result shows the relation between the two degree $-1$ Poisson algebras $(\cA, \{\cdot, \cdot\})$ and $(\mathfrak{X}^\bullet(\cN_{1, -1}), \{\cdot, \cdot \})$ -- it can be seen as a graded version of the Weinstein lagrangian neighborhood theorem.  

\begin{proposition}
The Poisson algebra $\cA \coloneqq C^{\infty}(\cN)$ is isomorphic to the Poisson algebra $\mathfrak{X}^\bullet(\cN_{1, -1})$ of multivector fields on the lagrangian submanifold $\cN_{1, -1} \subset \cN.$
\end{proposition}
\begin{proof}
Using (\ref{identification brst algebra and multivector fields}), we can write, for $n \in \bZ,$ 
\begin{equation}\label{degreewise identification brst algebra multivectorfields}
 C_n^{\infty}(\cN) = \bigoplus_{k \in \bZ} \cA^{k, n} \cong  \bigoplus_{k \in \bZ} \mathfrak{X}^{n - k, k }(\cN_{1, -1}) = \mathfrak{X}^{n}(\cN_{1, -1}),
\end{equation}
so  $C^{\infty}(\cN) $ and $ \mathfrak{X}^\bullet(\cN_{1, -1}) $ are isomorphic as graded vector spaces. 
But this identification also preserves the associative products; 
then, since the brackets of $C^{\infty}(\cN)$ and $ \mathfrak{X}^\bullet(\cN_{1, -1})$ coincide on generators, we conclude that (\ref{degreewise identification brst algebra multivectorfields}) extends to an isomorphism of degree $-1$ Poisson algebras.
\end{proof}

This proposition implies that the formal bivector $S^{(\infty)}$ on $\cN_{1, -1}$ also satisfies the master equation, that is, $\{S^{(\infty)}, S^{(\infty)}\} = 0.$ Then, if we put
\begin{equation}\label{1-ary and 2-ary derived bracket}
    \ell_{1}(f_1) \coloneqq \{Q, f_{1}\} \, \, \, \, \text{and}  \, \, \, \ell_{2}(f_1, f_2) \coloneqq (-1)^{f_1} \{\{\Pi, f_1\}, f_2 \} 
\end{equation}
and, in general, for $k \geq 3,$ we let 
$$\ell_{k}: C^{\infty}(\cN_{1, -1})^{\otimes k} \rightarrow C^{\infty}(\cN_{1, -1}) $$
be defined by
\begin{equation}\label{general k-ary derived bracket}
    \ell_{k}(f_1, \dots, f_{k}) \coloneqq (-1)^{\epsilon} \{\dots \{\{ \Pi^{(2-k)}, f_1\}, f_2\}, \dots, f_{k}\}
\end{equation}
where $\epsilon \coloneqq \sum_{i = 1}^{k} (k-i)f_i,$ we conclude, by Voronov's construction of homotopy algebras (\cite{Voronov}), that the extended BRST charge $S^{(\infty)}$ induces a homotopy Poisson structure on $C^{\infty}(\cN_{1, -1}),$ which shows that  the lagrangian $\cN_{1, -1} \subset \cN$ is a homotopy Poisson manifold. 
Alternatively, one can use the relations given by the master equation $\{S^{(\infty)}, S^{(\infty)}\} = 0$ to directly derive the coherence laws satisfied by the brackets (\ref{1-ary and 2-ary derived bracket}) and (\ref{general k-ary derived bracket}). For instance, from $\{Q, \Pi^{(-1)}\} + \frac{1}{2}\{\Pi, \Pi\} = 0,$ we obtain 

\begin{equation}\label{Jacobi up to homotopy 2-ary bracket}
    \sum_{\text{cyclic}}(-1)^{fh}\ell_2(f, \ell_2(g, h)) = (-)^{fh} (\ell_3 \circ \ell_1^{\otimes_3} + \ell_1 \circ \ell_3)(f, g, h)
\end{equation}
where $\ell_1^{\otimes_3}: C^{\infty}(\cN_{1, -1})^{\otimes_3} \rightarrow C^{\infty}(\cN_{1, -1})^{\otimes_3}$ is the differential on $C^{\infty}(\cN_{1, -1})^{\otimes_3}$ given by $$\ell_1^{\otimes_3}(f\otimes g \otimes h) \coloneqq \ell_1 (f) \otimes g \otimes h + (-1)^f f \otimes \ell_1(g) \otimes h + (-1)^{f +g} f \otimes g \otimes \ell_1(h).$$ Observe that (\ref{Jacobi up to homotopy 2-ary bracket}) says 
that the $2$-ary bracket defined by the cocycle $\Pi$ satisfies the Jacobi identity up to homotopy, where such a homotopy is given in terms of the $3$-ary bracket defined by $\Pi^{(-1)}.$

We conclude that the derived bracket (\ref{derived poisson bracket on H00}) on $H_{\{Q_{\ov\g},\, \cdot \, \}}^{0, 0}(\cA)$ is, therefore, just part of the
homotopy Poisson structure on the lagrangian submanifold $\cN_{1, -1},$ where such a structure is, in turn, fully encoded in the extended BRST charge $S^{(\infty)}.$  

\begin{remark}\label{remark on Voronov's construction}
To introduce the homotopy Poisson structure on $C^{\infty}(\cN_{1, -1})$ it is not necessary to identify $\cA \coloneqq C^{\infty}(\cN)$ with $\mathfrak{X}^\bullet(\cN_{1, -1}).$ In order to apply Voronov's construction of homotopy algebras (\cite{Voronov}),
it is enough to consider the abelian subalgebra $C^{\infty}(\cN_{1, -1}) \subset \cA$ and then use the extended BRST charge to define the higher derived brackets, since we already know that $\{S^{(\infty)}, S^{(\infty)}\} = 0$ on $\cA$ and this  suffices to directly compute the coherence laws for the derived brackets. Our motivation  to consider such an identification is to relate our construction with what was done in \cite{Sharapov} and to  connect it with the recent literature on homotopy Poisson manifolds and shifted Poisson structures (e.g. \cite{Behrend, Bandiera, Safronov, Pridham, Mehta}). From the viewpoint of \cite{Bandiera}, for instance, the structure we get on $C^\infty(\cN_{1, -1})$ is that of a degree 0 derived Poisson algebra (i.e., a $P_{\infty}$-algebra in the sense of Cattaneo-Felder  (\cite{Relative formality}) or a $\widehat{\mathbb{P}}_1$-algebra for Safronov (\cite{Safronov})).\par 
\end{remark}

When  $\pi \in \mathfrak{X}^2(M)$ is an actual Poisson structure, not just weak, we obtain the following  homological version of the generalized Marsden-Weinstein theorem by Cattaneo-Zambon (Theorem  \ref{generalized context - Marsden Weinstein thm}).

\begin{theorem}\label{homotopy version of the Kostant Sternberg result}
    Let $\psi: \g \rightarrow \mathfrak{X}(M)$ be a Lie algebra action on a Poisson manifold $(M, \pi),$ and let $J: M \rightarrow \h^*$ be a $\g$-equivariant map, for $\h$  a $\g$-module. Assume that the reduction data $(\psi, J)$ is regular and compatible with the Poisson structure $\pi.$  Then the $\bZ$-graded algebra $$\mathcal{K}^\bullet_{\g, \h} \coloneqq C^{\infty}(M) \otimes \bigwedge^\bullet \g^* \otimes \bigwedge^\bullet \h ,$$ graded by total ghost number, admits a homotopy Poisson structure with differential $\partial: \mathcal{K}^\bullet_{\g, \h} \rightarrow \mathcal{K}^{\bullet + 1}_{\g, \h}$ and a sequence of $k$-ary brackets $\{\cdot, \dots, \cdot\}_k,$ $k \geq 2,$ such that the Poisson algebra $(H_{\partial}^0(\mathcal{K^\bullet_{\g, \h}}), \{\cdot, \cdot\}_2)$ is identified with the reduced Poisson algebra $(C^{\infty}( C_{red}), \{\cdot, \cdot\}_{\pi_{red}}).$
\end{theorem}

This result follows from applying the degree one BFV construction to the moment map $J^{\sharp}: \ov\g[-1] \rightarrow C^{\infty}(\cM)$ defined by the generalized reduction data $(\psi, J)$. The regularity of the pair $(\psi, J)$ and its compatibility with the Poisson structure $\pi$ assures the existence of an extended BRST charge $S^{(\infty)}$. The algebra $\mathcal{K}^\bullet_{\g, \h}$ is then taken to be the abelian algebra $C^{\infty}(\cN_{1, -1}),$ whereas the differential and the higher brackets are those defined in terms of the extended BRST charge $S^{(\infty)}$. Finally, since $S^{(\infty)}$ satisfies the master equation $\{S^{(\infty)}, S^{(\infty)}\} = 0,$ it follows that these brackets turn $\mathcal{K}^\bullet_{\g, \h} \coloneqq C^{\infty}(\cN_{1, -1})$ into a homotopy Poisson algebra.
\subsection{The case of classical reduction data}\label{the case of the classical reduction data}
In the following, we show that Theorem \ref{homotopy version of the Kostant Sternberg result} applied to the classical hamiltonian reduction of Poisson manifolds recovers  the homological model obtained by Kostant-Sternberg (\cite{Kostant}), which we recalled in $\S$\ref{section brst in degree zero}.\par 

Let $(M, \pi)$ be a Poisson manifold endowed with a hamiltonian $\g$-action $\psi: \g \rightarrow \mathfrak{X}(M)$ with moment map $J: M \rightarrow \g^*.$ These data correspond to the hamiltonian action of  the graded Lie algebra $\ov\g = \g[1] \oplus \g$ on $\cM \coloneqq T^{*}[1]M$ with moment map $J^{\sharp}: \ov\g[-1] \rightarrow C^{\infty} (\cM)$ given by $J^{\sharp}_{1}(u) = u_M \in \mathfrak{X}(M),$ $u \in \g,$ and $J^{\sharp}_0(v) = J^*(v), v \in \g.$\par 
In this case, the geometric data encoded by the constraint submanifold $\cC = (J, J^\sharp)^{-1}(0)$ is the level set $C = J^{-1}(0)$ endowed with the tangent distribution $\cD = \langle J_1^\sharp(u) \rangle_{u \in \g} = \langle u_M \rangle_{u \in \g}.$
Under the hypothesis of the classical Marsden-Weinstein theorem, the degree one reduced space $\cC_{red}$ corresponding to the submanifold $\cC = (J, J^\sharp)^{-1}(0) \subset \cM$ exists and it is given by $\cC_{red} = T^*[1](J^{-1}(0)/G)$, for a $G$-action on $J^{-1}(0)$ integrating the infinitesimal $\g$-action, see Theorem \ref{marsden-weinstein cattaneozambon}.\par

Now, take $\{u^i\}$ a basis for the Lie algebra $\g.$ Let $\{u^{i, (0)}\}$ and $\{u^{j, (-1)}\}$ denote copies of this basis in the degree  0 and the degree $-1$ part of $\ov\g = \g[1] \oplus \g,$ respectively, and let $\{u_i^{(0)}\}$ and $\{u_j^{(-1)}\}$ denote the corresponding dual bases. With this notation, the BRST charge (\ref{BRST charge for degree one hamiltonian reduction}) reads 
\begin{equation}\label{deg one brst charge obtained from the classical reduction data}
Q_{\ov\g}  = J^{\sharp}_1(u^{i, (0)}) u^{(0)}_i + J^{\sharp}_0(u^{j, (-1)})  u_j^{(-1)} - \frac{1}{2} c_k^{ij} \, u^{(0)}_i \, u^{(0)}_j \, u^{k, (0)} - c^{mn}_{p} u^{(0)}_{m} \, u^{(-1)}_{n} \, u^{p, (-1)}.
\end{equation}

Since $\ov\g = \g[1] \oplus \g$ is a degree one Lie algebra, it follows that $\{Q_{\ov\g}, Q_{\ov\g}\} = 0;~$   hence, we may consider the differential graded algebra $(\cA, \{Q_{\ov\g}, \, \cdot \, \})$ (here, as in the general case, $\cA \coloneqq C^{\infty}(\cN),$ where $\cN \coloneqq T^*[1](M \times \ov\g^*[-1])).$ \par  
Therefore, from Theorem \ref{main result}, we conclude that $$ H^{0, \bullet}_{\{Q_{\ov\g},\, \cdot \, \}}(\cA) \cong \mathfrak{X}^\bullet(\cC_{red}),$$ that is, the degree zero cohomology of the complex $(\cA, \{Q_{\ov\g}, \, \cdot \, \})$ is isomorphic to the degree $-1$ Poisson algebra of multivector fields on the reduced space $J^{-1}(0)/G.$ \par 
 
The Poisson structure $\pi \in \mathfrak{X}^2(M)$ is clearly compatible with the pair $(\psi, J),$  
so Corollary \ref{cohomological condition for the reducibility of pi} guarantees that there exists a cohomology class $[\Pi] \in H_{\{Q_{\ov\g}, \, \cdot \, \}}^{0, 2}$ for which $\Pi^{0, 0} = \pi.$ Using that the degree one moment map $(J, J^\sharp)$ was obtained from the moment map $J: M \rightarrow \g^*,$ we compute
\begin{align*}
\{Q_{\ov\g}, \pi \} &  = \{J_1^\sharp(u^{i, (0)}), \pi \}u^{(0)}_i + \{J_0^\sharp(u^{j, (-1)}), \pi \}u_j^{(-1)}\\
& = \{u^i_M , \pi \}u_i^{(0)} + \{J^*u^j, \pi \}u_j^{(-1)}\\
& = \{ \pi, \{\pi, J^*u^i \}\}u_i^{(0)} + \{\pi, J^* u^j\}u_j^{(-1)}\\
& = - u^j_M u^{(-1)}_j.
\end{align*}
It follows that
\begin{equation}\label{explicit representative for the cohomology class related to the Poisson structure - classical case}
\Pi \coloneqq \pi + u^{i, (0)} u^{(-1)}_i
\end{equation}
satisfies $\{Q_{\ov\g}, \Pi\} = 0,$ so it  can be taken as an explicit representative for the cohomology class encoding the reducible Poisson structure $\pi \in \mathfrak{X}^2(M).$\par  

Now, notice that, for $Q_{\ov\g}$ is as in (\ref{deg one brst charge obtained from the classical reduction data}) and $\Pi$ is as in (\ref{explicit representative for the cohomology class related to the Poisson structure - classical case}),  we have $$\{Q_{\ov\g} + \Pi, Q_{\ov\g} + \Pi \} = 0.$$  
So, in this case, the extended BRST charge is given as a sum of only two terms, 
\begin{equation}\label{extended brst charge for lifted hamiltonian action}
S^{(\infty)} = Q_{\ov\g}+ \Pi.
\end{equation}
 
Hence, in the homotopy Poisson structure induced by (\ref{extended brst charge for lifted hamiltonian action}) on $C^{\infty}(\cN_{1, -1})$ only the first two brackets are  nonzero, which means that $C^{\infty}(\cN_{1, -1})$ is simply a differential Poisson algebra. Using (\ref{1-ary and 2-ary derived bracket}), one can directly verify that this structure coincides with the one 
described in $\S$\ref{section brst in degree zero}.

\subsection{Hamiltonian actions of dgla's} 
What we have just described turns out to be a particular instance of a more general construction. Indeed, let $(\ov\g \coloneqq \h[1] \oplus \g, \delta)$ be a differential graded Lie algebra (dgla) concentrated in degrees $-1$ and $0.$ Then, for $(M, \pi)$ a Poisson manifold, let $\Psi: (\ov\g, \delta) \rightarrow (\mathfrak{X}(\cM), [X_{\pi}, \cdot \, ])$ be a morphism of dgla's, where $X_{\pi} \coloneqq \{\pi, \, \cdot \, \},$ and let $J^{\sharp}: \ov\g[-1] \rightarrow C^\infty(\cM)$ be a moment map for this $\ov\g$-action.

In this case, we have
\begin{align}
    \{\pi, J^\sharp_1(u)\} & = 0,  \hspace{0.2cm} u \in \g, \label{eq: poisson vfs dgla case}\\
    \{\pi, J^\sharp_0(v)\} & = J^\sharp_1(\delta(v)), \hspace{0.2cm} v \in \h.\label{eq: hamiltonian vfs dgla case}
\end{align}
So $\pi \in N(\cJ).$ For the BRST charge $Q_{\ov\g}$ as in (\ref{BRST charge for degree one hamiltonian reduction}), using \eqref{eq: poisson vfs dgla case} and \eqref{eq: hamiltonian vfs dgla case}, we obtain 
\begin{equation}\label{eq: Q pi dgla case} 
        \{Q_{\ov\g}, \pi\}  = \{ J^\sharp_0(v^j)v^*_j, \pi\}  = \{\pi, J^\sharp_0(v^j)\}v^*_j = J^\sharp_1(\delta(v^j))v^*_j.
\end{equation}

Now, let 
\begin{equation}\label{eq: Pi cocycle dgla case}
    \Pi \coloneqq \pi - a^i_j v^*_i u^j,
\end{equation}
where $ A\coloneqq (a^i_j)$
is the matrix of $\delta: \h \rightarrow \g$ in the bases $\{u^i\}$ of $\g$ and $\{v^j\}$ of $\h$. Then, from \eqref{eq: Q pi dgla case} and the fact that $\delta: \h \rightarrow \g$ is derivation, it follows that 
$$ 
S^{(\infty)} = Q_{\ov\g} + \Pi
$$ 
satisfies $\{S^{(\infty)}, S^{(\infty)}\} = 0.$
This extended BRST charge induces the structure of differential graded Poisson algebra on $$\mathcal{K}_{\g, \h} = C^{\infty}(M) \otimes \bigwedge^\bullet \g^* \otimes \bigwedge^\bullet \h.$$\par 
 It is clear that for the graded algebra $\ov\g \coloneqq \g[1]\oplus \g$ with differential $\delta \coloneqq -id: \g \rightarrow \g,$ the above construction recovers the BFV algebra of the usual hamiltonian case  (c.f. $\S$\ref{the case of the classical reduction data}).

\subsection{Lie bialgebra and quasi-bialgebra actions}
We will now look at the BFV model in the context of Lie bialgebra and quasi-bialgebra actions.
For the former, we are still in the setup of Theorem \ref{homotopy version of the Kostant Sternberg result}, whereas, for the latter, we go back to the more general case of weak Poisson structures ($\S$\ref{subsection: reducible Poisson str}).

Recall that a Lie bialgebra is a Lie algebra $\g$ endowed with a $1$-cocycle $F: \g \rightarrow \bigwedge^2 \g$ for the adjoint representation of $\g$ on $\bigwedge^2 \g$ whose dual defines a Lie algebra structure on $\g^*.$

An action of the Lie bialgebra  $(\g, F)$ on a Poisson manifold $(M, \pi)$ is a $\g$-action $ \psi: \g \rightarrow \mathfrak{X}(M)$ for which 
\begin{equation}\label{lack of inv of pi - bialgebra}
    \{\pi, u^k_M\} = (F(u^k))_M = a^k_{ij} u^i_M u^j_M,
\end{equation}
 where we let $\psi(u^i) \coloneqq u^i_M,$ as usual, and $\{a^k_{ij}\}$ are the structural constants of the dual bracket with respect to the dual basis of $\{u^i\}.$ Notice that equation (\ref{lack of inv of pi - bialgebra}) says that the Poisson structure $\pi$ is not necessarily invariant with respect to the $\g$-action, but the lack of invariance is controlled by the cobracket; in graded terms, if we let  $\cJ$ denote the homogeneous ideal generated by $\{u^i_M\}$,  this condition ensures that $\pi \in N(\cJ)$. From the perspective of Definition \ref{compatibility of the Poisson struct with the reduction data (psi, J)}, we have a pair $(\psi, J)$ compatible with the Poisson structure $\pi \in \mathfrak{X}^2(M),$ where $J \equiv 0$ $(\h = \{0\}).$

 Since we have only the degree one constraints, which are given by the infinitesimal $\g$-action $\psi: \g \rightarrow \mathfrak{X}(M),$ the BRST charge can be written as
\begin{equation}\label{BRST charge - bialgebra}
    Q_{\g} = u^i_M u^*_i - \frac{1}{2} c^{ij}_k u^*_i u^*_j u^k.\\
\end{equation}

Since $\pi \in N(\cJ)$, it follows from Corollary \ref{cohomological condition for the reducibility of pi} that there exists a cohomology class $[\Pi] \in H^{0, 2}_{\{Q_{\g}, \cdot \}}(\cA)$ associated to the Poisson structure $\pi \in \mathfrak{X}^2(M).$ In this case, we can also write down such a cocycle explicitly. \par
For this, recall that, if $[\cdot, \cdot]$ denotes the Lie bracket on $\g$ and $[\cdot, \cdot]^*$ denotes the Lie bracket on $\g^*,$ 
then, for $u, v\in \g$ and $\xi, \eta \in \g^*, $ we have (\cite[Prop.~10.11]{Vaisman})
\begin{equation}\label{bracket cobracket equation}
    [\xi, \eta]^*([u, v]) = -[ad^*_u \xi, \eta]^*(v) - [\xi, ad^*_u\eta]^*(v) + [ad^*_v \xi, \eta]^*(u) + [\xi, ad^*_v \eta]^*(u).
\end{equation}
Then, from (\ref{bracket cobracket equation}), we obtain, for fixed basis elements $u_i, u_j \in \g$ and $u^*_m, u^*_n \in \g^*,$ the following equation in terms of the structural constants
\begin{equation}\label{bracket-dual bracket relation: structural constants}
    a_{ij}^{\ell}c_{\ell}^{mn} = - a_{\ell j}^n c^{\ell m}_{i} - a_{i\ell}^n c^{\ell m}_{j} + a_{\ell j}^m c^{\ell n}_{i} + a_{i\ell}^m c^{\ell n}_{j}.
\end{equation}
Using (\ref{lack of inv of pi - bialgebra}) and (\ref{bracket-dual bracket relation: structural constants}), a direct computation shows that
\begin{equation}\label{Poisson cocycle - bialgebra}
    \Pi  = \pi + a^j_{ik} u^i_M u^*_j u^k
\end{equation}
is a $\{Q_{\g}, \cdot\}$-cocyle, $\{Q_\g, \Pi\} = 0.$ One can see that $\{\Pi, \Pi\}$ is not necessarily zero, so the extended BRST charge  will contain higher terms, 
\begin{equation}\label{eq: extended charge bialgebra case}
    S^{(\infty)} \coloneqq Q_{\g} + \Pi + \sum_{k \geq 1} \Pi^{(-k)},
\end{equation} 
for $Q_{\g}$ and $\Pi$ as in (\ref{BRST charge - bialgebra}) and (\ref{Poisson cocycle - bialgebra}), respectively (c.f. \cite[Prop. 21]{AleZabzine}). From the construction in $\S$\ref{subsection: homotopy structure}, we obtain the following result.

\begin{proposition}\label{proposition: homotopy strc bialgebra}
    Let $(M, \pi)$ be a Poisson manifold, and let $\psi: \g \rightarrow \mathfrak{X}(M)$ be an action of a Lie bialgebra $(\g, F).$ Then extended charge (\ref{eq: extended charge bialgebra case}) induces the structure of homotopy Poisson algebra on the graded algebra
$$\mathcal{K_{\g}} \coloneqq C^{\infty}(M) \otimes \bigwedge^\bullet \g^*.$$ In particular, the corresponding $2$-ary bracket is given by \begin{equation}\label{f, g bracket - bialgebra}
    \ell_2(f, g) = \{ \{\pi, f \}, g \}, \hspace{0.2cm} f, g \in C^{\infty}(M),
\end{equation}
\begin{equation}\label{f, u* bracket - bialgebra }
\ell_2(f, u^*) = \{u_M^i, f\}[u^*_i, u^*]^*, \hspace{0.2cm}  f \in C^{\infty}(M), \, \, u^* \in \g^*,
\end{equation} 
\begin{equation}\label{u* u* bracket - bialgebra}
    \ell_2(u_1^*, u_2^*) = 0, \hspace{0.2cm} u_1^*, u_2^* \in \g^*.
\end{equation}
\end{proposition}

From the master equation satisfied by the extended BRST charge $S^{(\infty)},$ we obtain, in particular, the relation
\begin{equation}\label{homotopy relation bialgebra case}
    \{Q, \Pi^{(-1)}\} + \frac{1}{2}\{\Pi, \Pi\} = 0,
\end{equation}
 which says that $\Pi^{-1}$ is the term generating the homotopy for the Jacobiator of the brackets defined in (\ref{f, g bracket - bialgebra}), (\ref{f, u* bracket - bialgebra }) and (\ref{u* u* bracket - bialgebra}). In the present case, since $\pi$ is strictly Poisson, we have   $\Pi^{-1} \in A^{1,2}\oplus A^{2,3}$: the $A^{1, 2}$-term generates the homotopy for the Jacobiator of $f, g \in C^{\infty}(M)$ and $u^* \in \g^*,$ whereas the $A^{2, 3}$-term gives the homotopy for the Jacobiator of $f \in C^{\infty}(M)$ and $u^*_1, u^*_2 \in \g^*.$

This discussion extends to the case of actions of quasi-bialgebras.
Recall that a Lie quasi-bialgebra is a triple $(\g, F, \chi),$ where $\g$ is a Lie algebra, $F: \g \rightarrow \bigwedge^2 \g$ is a cocycle for the adjoint representation of $\g$ on $\bigwedge^2 \g,$ and $\chi \in \bigwedge^3 \g,$ for which the brackets
\begin{equation*}
    [(u, 0), (v, 0)] = ([u, v], 0), \hspace{0.2cm} u, v \in \g,
\end{equation*}
\begin{equation*}
    [(u, 0), (0, v^*)] = (\iota_{v^*} F(u), \text{ad}^*_{u}v^*), \hspace{0.2cm} u \in \g, \, \, v^* \in \g^*,
\end{equation*}
 \begin{equation*}
    [(0, u^*), (0, v^*)] = (\chi(u^*, v^*), F^*(u^*, v^*) ), \hspace{0.2cm} u^*, v^* \in \g^*,
\end{equation*}
define a Lie algebra structure on $\g \oplus \g^*.$ \par

Given a quasi-Lie bialgebra $(\g, F, \chi),$ a {\em quasi-Poisson $\g$-space} is a manifold $M$ endowed an $\g$-action $\psi: \g \rightarrow \mathfrak{X}(M)$ and a bivector field $\pi \in \mathfrak{X}^2(M)$ satisfying  

\begin{equation}\label{quasibialgebra invariance of pi}
    \{\pi, u^k_M\} = (F(u^k))_M = a^k_{ij} u^i_M u^j_M,
\end{equation}
\begin{equation}\label{non integrability - quasi bialgebra}
    \frac{1}{2}\{\pi, \pi\} = \chi_M.
\end{equation}

If $\cJ$ denotes the homogeneous ideal spanned by the degree one constraints associated to the infinitesimal action $\psi: \g \rightarrow \mathfrak{X}(M),$ we see that the condition $\pi \in N(\cJ)$ follows from \eqref{quasibialgebra invariance of pi}, whereas (\ref{non integrability - quasi bialgebra}) says that $\pi \in \mathfrak{X}^2(M)$ is a weak Poisson structure.\par 
In this setting, the first two terms of the corresponding extended BRST charge are also given by (\ref{BRST charge - bialgebra}) and (\ref{Poisson cocycle - bialgebra}). The additional piece of data $\chi \in \bigwedge^3 \g$ enters as the $A^{0, 1}$-term of $\Pi^{-1}$ that did not appear in the bialgebra case, that is, we will have 

\begin{equation}\label{term of tgh -1 - quasi bialgebra}
     \Pi^{(-1)} = \frac{1}{3}\psi(\iota_{u^*_i} \chi)u^i + \cdots. 
\end{equation}

As before, the extended BRST charge 
\begin{equation}
    S^{(\infty)} \coloneqq Q_{\g} + \Pi + \sum_{k \geq 1} \Pi^{(-k)},
\end{equation} 
induces the structure of homotopy Poisson algebra on the graded algebra 
$$\mathcal{K_{\g}} \coloneqq C^{\infty}(M) \otimes \bigwedge^\bullet \g^*.$$ Since $\Pi$ remains the same, the $2$-ary bracktes are still given by (\ref{f, g bracket - bialgebra}), (\ref{f, u* bracket - bialgebra }) and (\ref{u* u* bracket - bialgebra}). However, in this case, the Jacobiator of three functions $f, g, h \in C^{\infty}(M)$ vanishes only up to homotopy, and such a homotopy depends only on the $A^{0, 1}$-term of $\Pi^{-1}$ which is given in terms of the trivector $\chi \in \bigwedge^3 \g$.

\subsection{Hamiltonian quasi-Poisson manifolds} Recall that a $G$-manifold $M$ endowed with a \textit{invariant bivector} $\pi \in \mathfrak{X}^2(M)$ is said to be a hamiltonian quasi-Poisson space (\cite{Alekseev}) if 
\begin{equation}\label{non-int of pi hamiltonian quasi-Poisson}
\{\pi, \pi\} = \phi_M \coloneqq \frac{1}{12} \sum_{i, j, k}\langle u^i, [u^j, u^k] \rangle (u^i \wedge u^j \wedge u^j)_M,
\end{equation}
where  $\langle \cdot, \cdot  \rangle$ denotes an invariant inner product on the Lie algebra $\g \coloneqq \text{Lie} (G)$ and $\{u^i\}$ is an orthonormal basis for it, and if there exists an equivariant map $\Phi: M \rightarrow G$ for which we have the moment map condition \begin{equation}\label{G-valued moment condition} 
\pi^{\sharp}(d(\Phi^*f)) = \frac{1}{2} \sum_k \Phi^*\big ( ( u^k_L + u^k_R )f\big) u^k_M, \hspace{0.2cm} f \in C^{\infty}(G),\end{equation}
where $u^k_L$ and $u^k_R$ are, respectively, the left and right-invariant vector fields associated to $u^k \in \g.$\par 

In this context, 
if the identity $e \in G$ is a regular value of $\Phi: M \rightarrow G$ and the action of $G$ along $\Phi^{-1}(e)$ is free and proper, then the quotient $\Phi^{-1}(e)/G$  inherits a  Poisson structure (\cite[Thm.~6.1]{Alekseev}). In the following, we obtain a homological version of this result.

Let $\mathcal{U} \subset G$ be an open neighborhood of the identity $e \in G$ where $exp: \g \rightarrow G$ is a diffeomorphism, and let $log: \mathcal{U} \rightarrow \g$ denote its local inverse. Then, for each $u^k \in \g, $ define $f^k: \mathcal{U} \rightarrow \mathbb{R} $ by $$f^k(g) \coloneqq \langle \, u^k, log(g) \, \rangle, \hspace{0.2cm} g \in \mathcal{U}. $$ Taking the pullback by the moment map $\Phi: M \rightarrow G,$ we obtain $\Phi^{*}f^k: \Phi^{-1}(\mathcal{U}) \rightarrow \mathbb{R}$ defined in the open neighborhood $\Phi^{-1}(\mathcal{U})$ of $\Phi^{-1}(e).$ From now on, we shall assume that the map $\Phi: M \rightarrow G $ has its image cointained in the neighborhood $\mathcal{U},$ so that  $\Phi^{*}f^k$ is actually a globally defined function on $M$; in general, we can just replace $M $ by $\Phi^{-1}(\mathcal{U}).$ \par 

The vanishing set of the ideal of $C^{\infty}(M)$ spanned by $\{\Phi^*{f^k}\}$ coincides with the level set $\Phi^{-1}(e)$ and,  moreover, we have the following.

\begin{lemma}\label{prop: forms LI along the level of G-moment}
    If $e \in G$ is a regular value of $\Phi: M \rightarrow G,$ then the $1$-forms $\{d(\Phi^*f^k)\}$ are linearly independent along $\Phi^{-1}(e).$
\end{lemma}
\begin{proof}
Notice that 
\begin{align}
(df^i)_e(u^j) & = \frac{d}{dt}\bigg|_{t = 0} f^i(exp(tu^j)) \label{prop preimage of e} \\
& = \frac{d}{dt}\bigg|_{t = 0} \langle \, u^i, log(exp(tu^j)) \, \rangle = \frac{d}{dt}\bigg|_{t = 0} t\langle u^i, u^j \, \rangle = \delta^{ij} \nonumber.
\end{align}
Then, assume that $\sum_k \alpha_k \,  d(\Phi^*(f^k)) \equiv 0$ along $\Phi^{-1}(e).$ Fix $m \in \Phi^{-1}(e).$ Since $e \in G$ is a regular value, for each $u^k,$ there exists $X_k \in T_{m} M $ satisfying $T_m \Phi (X_k) = u^k.$ From this and (\ref{prop preimage of e}), one easily concludes that $\alpha_k \equiv 0$ along $\Phi^{-1}(e).$
\end{proof}
Now, let  $\cJ$ denote the homogeneous ideal of $C^{\infty}(\cM)$ generated by $\{\Phi^{*}(f^k)\}$ (degree zero) and $\{u^k_M\}$ (degre one). Since $\{u^i_M\}$ are fundamental vector fields of a $G$-action,  it follows  that $\{u^i_M, u^j_M\} = c^{ij}_k u^k_M.$ We claim that $(u^i)_M(\Phi^*f^j) =  c^{ij}_k \Phi^*(f^k).$
Indeed, notice that
\begin{align*}
    u^i_M(\Phi^*f^j) & = \frac{d}{dt}\bigg|_{t = 0} \langle \,  u^j, log\big(\Phi(exp(-tu^i)m) \big) \, \rangle \\
    & = \langle \,  u^j, \frac{d}{dt}\bigg|_{t = 0}  log\big(\Phi(exp(-tu^i)m) \big) \,  \rangle;
\end{align*}
however, since $\Phi: M \rightarrow G$ is equivariant, we have $$\Phi(exp(-tu^i)m) = exp(-tu^i) \Phi(m) exp(tu^i),$$
from which we obtain $$ \frac{d}{dt}\bigg|_{t = 0} \Phi(exp(-tu^i)m) =  \frac{d}{dt}\bigg|_{t = 0}   exp(-tu^i) \Phi(m) exp(tu^i) = u_L^i - u_R^i.$$
Hence, we see that 
\begin{equation} \label{derivative in terms of Tlog}
     (u^i)_M(\Phi^*f^j) = \langle \,  u_j, ( T_{\Phi(m)} \, log) (u_L^i - u_R^i) \, \rangle.
\end{equation}
Then, all we need is to compute $( T_{\Phi(m)} \, log) (u_L^i - u_R^i)$ more explicitly. For this, we shall use the Baker-Campbell-Hausdorff (BCH) formula. In fact, on the Lie group $G,$ we have $$ u_R^i (\Phi(m)) = \frac{d}{dt}\bigg|_{t = 0} exp(tu^i)\Phi(m) \hspace{0.2cm} \text{and} \hspace{0.2cm} u_L^i (\Phi(m)) = \frac{d}{dt}\bigg|_{t = 0} \Phi(m)exp(tu^i);$$ since we are assuming that $\Phi(m) \in G$ lies in a neighbourhood of the identity where $\exp : \g \rightarrow G$ is a diffeomorphism, we obtain that there exists a unique $u_{\Phi(m)}$ for which $exp(u_{\Phi(m)}) = \Phi(m)$ (that is, $log(\Phi(m)) = u_{\Phi(m)} $). Therefore, by the BCH formula,  we have that 
$$u_R^i (\Phi(m)) = \frac{d}{dt}\bigg|_{t = 0} exp(tu^i)\Phi(m) = \frac{d}{dt}\bigg|_{t = 0} exp(z^{tu^i}_{u_{\Phi(m)}}), $$
where \begin{equation}\label{z(tui, uPhi)}
    z^{tu^i}_{u_\Phi(m)} = tu^i + u_{\Phi(m)} + \frac{t}{2} [u^i, u_{\Phi(m)}] + \frac{t}{12} [u_{\Phi(m)}, [u_{\Phi(m)}, u^i]] - \frac{t}{720} ad^4_{u_{\Phi(m)}}(u^i) + \cdots;
\end{equation} similarly for $u_L^i(\Phi(m)),$ but with \begin{equation}\label{z(uPhi, ui)}
    z^{u_{\Phi(m)}}_{tu^i} = u_{\Phi(m)}  + tu^i + \frac{t}{2}[u_{\Phi(m)}, u^i] + \frac{t}{12}[u_{\Phi(m)}, [u_{\Phi(m)}, u^i]] - \frac{t}{720} ad^4_{u_{\Phi(m)}}(u^i) + \cdots.
\end{equation}
Using (\ref{z(tui, uPhi)}) and (\ref{z(uPhi, ui)}), we compute 
\begin{align*}
   (T_{\Phi(m)} log )(u_L^i - u_R^i) & = \frac{d}{dt}\bigg|_{t = 0} \bigg(  log \big( exp(z_{tu^i}^{u_\Phi(m)}) \big) - log \big( exp(z^{tu^i}_{u_\Phi(m)}) \big)   \bigg) \\
   & = \frac{d}{dt}\bigg|_{t = 0} \big( z_{tu^i}^{u_\Phi(m)} -  z^{tu^i}_{u_\Phi(m)} \big)
   =  [u_{\Phi(m)}, u^i].
\end{align*}
Going back to (\ref{derivative in terms of Tlog}), we conclude that 
\begin{align*}
    (u^i)_M(\Phi^*f^j) & = \langle \,  u^j, ( T_{\Phi(m)} \, log) (u_L^i - u_R^i) \, \rangle \\
    & = \langle \, u^j  , [u_{\Phi(m)}, u^i] \, \rangle\\
    & =  \langle \, [u^i, u^j], u_{\Phi(m)} \, \rangle \\
    & =  \langle \, [u^i, u^j],\,  log(\Phi(m)) \, \rangle =  c^{ij}_k \langle \, u^k, \, log(\Phi(m)) \, \rangle =  c^{ij}_k \Phi^*(f^k),
\end{align*}
as we have claimed.
Therefore,  the constraints generating $\cJ$ give a moment map $J^{\sharp}: \ov\g[-1] \rightarrow C^{\infty}(\cM).$ 
It follows that the BRST charge associated to the ideal $\cJ$ is given by $$ Q = u^i_M u^{(0)}_i + \Phi^*(f^j)u^{(-1)}_j - \frac{1}{2}c^{ij}_k u_i^{(0)}u_j^{(0)} u^{k, (0)} - c^{mn}_p u^{(0)}_m u^{(-1)}_n u^{p, (-1)}.$$
\begin{remark}
    An alternative way to verify that $(u^i)_M (\Phi^* f^j) = c^{ij}_k \Phi^*(f^k)$ is to use the fact that there exists a Poisson bivector $\pi_0 \in \mathfrak{X}^2(M)$ for which the map $\mu \coloneqq log \circ \Phi: M \rightarrow \g^* \cong \g$ is a usual moment map for the $G$-action on $M$ (\cite[Cor. 7.3]{Alekseev}). In this case, since $\mu^* u^ k = \Phi^*f^k,$ for $u^k \in C^\infty(\g^*),$ it follows that 
    \begin{align*}
        (u^i)_M (\Phi^* f^j) & = \pi^\sharp_0(d(\mu^*u^i))(\mu^* u^j) \\
        & = \pi_0( d(\mu^*u^i), d(\mu^*u^j) ) \\
        & = \{\mu^*u^i, \mu^*u^j\}_{\pi_0} \\
        & = \mu^*(\{u^i, u^j\}_{\g^*}) = c^{ij}_k \mu^*u^k = c^{ij}_k \Phi^*f^k.
    \end{align*}
\end{remark}

Notice that the moment map condition (\ref{G-valued moment condition}) implies that $\pi \in N(\cJ),$ so the quasi-Poisson structure $\pi$ is compatible with the generalized hamiltonian reduction data associated to the ideal $\cJ.$ On the other hand, the freeness assumption on the $G$-action and Lemma  \ref{prop: forms LI along the level of G-moment} imply that $0 \in \g^*$ is regular value of the moment map $(J, J^\sharp): \cM \rightarrow (\ov\g[-1])^*,$ which assures the existence of a corresponding extended BRST charge ($\S$\ref{subsection: extended brst charge}). Thus, from $\S$\ref{subsection: homotopy structure}, we obtain the following result.

\begin{proposition}
    Let $(M, \pi)$ be a hamiltonian quasi-Poisson space with moment map $\Phi: M \rightarrow G.$ Assume that the identity $e \in G$ is a regular value of $\Phi: M \rightarrow G$ and that the $G$-action on $\Phi^{-1}(e)$ is free and proper. Let $\cU \subset G$ be neighborhood of $e \in G$ where $exp: \g \rightarrow G$ is diffeomorphism and set $M_{\cU} \coloneqq \Phi^{-1}(\cU).$ Then the algebra $$\mathcal{K}^\bullet_{\g, \g} \coloneqq C^{\infty}(M_{\cU}) \otimes \bigwedge^\bullet \g^* \otimes \bigwedge^\bullet \g$$ admits a homotopy Poisson structure for which we have
    $$(H_{\partial}^0(\mathcal{K^\bullet_{\g, \g}}), \{\cdot, \cdot\}_2) \cong \bigg( C^{\infty}\bigg( \frac{\Phi^{-1}(e)}{G}\bigg), \pi_{red} \bigg).$$
\end{proposition}

This result provides a BFV model for  hamiltonian quasi-Poisson reduction (\cite[Thm.~6.1]{Alekseev}).\par


\begin{thebibliography}{1}


\bibitem{Alekseev} A. Alekseev, Y. Kosmann-Schwarzbach, E. Meinrenken, Quasi-Poisson manifolds, Canadian J. of Math.,  \textbf{54} , Issue 1, 3-29, 2002

\bibitem{Alekseev2} A. Alekseev, Y. Kosmann-Schwarzbach, Manin pairs and moment maps, J. Differential Geometry, \textbf{56}, 136-165, 2000.

\bibitem{Bandiera} R. Bandiera, Z. Chen, M. Sti\'enon, P. Xu, Shifted derived Poisson manifolds associated with Lie pairs, 	Comm. Math. Phys., 	\textbf{375}, 1717-1760, 2020.


\bibitem{BF} I.A. Batalin, E.S. Fradkin, A generalized canonical formalism and quantization of reducible
gauge theories, Physics Letter \textbf{122B}, 157-164, 1983.

\bibitem{BV} I.A. Batalin, G.A. Vilkovisky, Relativistic S-matrix of dynamical systems with bosons and fermions, Physics Letters \textbf{69B}, 309-312, 1977.

\bibitem{BV1} I.A. Batalin, G.A. Vilkovisky, Gauge algebra and quantization, Phys. Lett. B \textbf{102}, no. 1, 27-31, 1981.

\bibitem{BV2} I. A. Batalin, G. A. Vilkovisky Quantization of gauge theories with linearly dependent generators, Physical Review D \textbf{28}, no. 10, 1983.

\bibitem{BV3} I.A. Batalin, G.A. Vilkovisky, Existence theorem for gauge algebra, Journal of Math. Physics, \textbf{26}, 172-184, 1985.


\bibitem{BRS} C. Becchi, A. Rouet, R. Stora, The abelian Higgs Kibble model, unitarity of the S-operator,
Physics Letters, \textbf{52B}, 344-346, 1974.

\bibitem{BRS1} C. Becchi, A. Rouet, R. Stora, Renormalization of the abelian Higgs-Kibble model, Commun. Math. Physics, \textbf{42}, 127-162, 1975.

\bibitem{BRS2}  C. Becchi, A. Rouet, R. Stora, Renormalization of gauge theories, Annals of Physics, \textbf{98}, 287-321, 1976.


\bibitem{Behrend} K. Behrend, M. Peddie, P. Xu, Quantization of $(-1)$-shifted derived Poisson manifolds, Commun. Math. Physics, {\bf 402}, 2301-2338, 2023. 

\bibitem{AleZabzine} F. Bonechi, A. Cabrera, M. Zabzine, AKSZ construction from reduction data, JHEP07(2012)068. 

\bibitem{GGGR} H. Bursztyn, A. S. Cattaneo, R. A. Mehta, M. Zambon, Graded geometry and generalized reduction, arXiv:2306.01508, 2023.

\bibitem{Henrique} H. Bursztyn, G. Cavalcanti, M. Gualtieri, Reduction of Courant algebroids and generalized complex strucutures, Advances in Math. \textbf{211}, 726-765. 


\bibitem{Relative formality} A. S. Cattaneo, G. Felder, Relative formality theorem and quantization of coisotropic submanifolds, Adv. Math., \textbf{208}, no. 2, 521-548, 2007.

\bibitem{CattaneoZambon} A. S. Cattaneo, M. Zambon, A supergeometric approach to Poisson reduction, Comm. in Math. Physics, \textbf{318} (3), 675-716, 2013.


\bibitem{Dufour} J-P. Dufour, N. T. Zung, Poisson structures and their normal forms, Progress in Mathematics, Birkh\"ause Verlag, vol. 242, 2005.
 
\bibitem{FalcetoZambon} F. Falceto, M. Zambon, An extension of the Marsden-Ratiu reduction for Poisson manifolds, Lett. Math. Physics, vol. {\bf 85}, 203-219, 2008.

\bibitem{Stasheffetal} J. Fisch, M. Henneaux, J. Stasheff, C. Teitelboim, Existence, uniqueness, and cohomology of the classical BRST charge with ghosts of ghosts, Comm. Math. Physics, \textbf{120}, 379-407, 1989. 



\bibitem{Henneaux's Book} M. Henneaux, C. Teitelboim, Quantization of Gauge Systems, Princeton University Press, 1992.

\bibitem{Kimura} T. Kimura, Generalized classical BRST cohomology and reduction of Poisson Manifolds, Comm. in Math. Physics, \textbf{151}, 155-182, 1993.

\bibitem{Kostant} B. Kostant, S. Sternberg, Symplectic reduction, BRS cohomology, and infinite dimensional Clifford algebras, Annals of Physics, 49-113, \textbf{176}, 1987.

\bibitem{Lu} J.-H. Lu, A. Weinstein, Poisson Lie groups, dressing transformations, and Bruhat decompositions, J. Differential Geometry, \textbf{31},501-526, 1990.

\bibitem{Sharapov} S. L. Lyakhovich, A. A. Sharapov, BRST theory without Hamiltonian and Lagrangian, JHEP03(2005)011. 

\bibitem{MarsdenRatiu} J. E. Marsden, T. Ratiu,  Reduction of Poisson manifolds, Letters in Math. Physics, \textbf{11}, 161-169, 1986.

\bibitem{MarsdenWeinstein} J. E. Marsden, A. Weinstein, Reduction of symplectic manifolds with symmetry, Rep. Math. Phys., \textbf{5}, 121-130. 1974.

\bibitem{Mehta} R. A. Mehta, On homotopy Poisson actions and reduction of symplectic $Q-$manifolds, Diff. Geo. Appl., \textbf{29}, 319-328, 2011.



\bibitem{Pridham} J.P. Pridham, An outline of shifted Poisson structures and deformation quantisation in derived differential geometry, arXiv:1804.07622, 2018.


\bibitem{Roytenberg} D. Roytenberg, On the structure of graded symplectic supermanifolds and Courant algebroids, Quantization, Poisson brackets and beyond (Manchester 2001), volume 315, Comtep. Math., 169-185, Amer. Math. Soc., Providence, RI, 2002.


\bibitem{Safronov} P. Safronov, Poisson reduction as coisotropic intersection, Higher Structures, {\bf 1}, 87-121, 2017.


\bibitem{Stasheff} J. Stasheff, Homological reduction of constrained Poisson algebras, J. Differential Geometry \textbf{45},  221-240, 1997.

\bibitem{Vaisman} I. Vaisman, Lectures on the geometry of Poisson manifolds, Progress in Mathematics, Birkh\"ause Verlag, 1994.

\bibitem{Voronov} T. Voronov, Higher derived brackets and homotopy algebras, J. Pure Appl. Algebra, \textbf{209}, 133-153, 2005.

\end{thebibliography}
\end{document}